\documentclass[final]{siamart190516}

\usepackage[letterpaper,top=2cm,bottom=2cm,left=3cm,right=3cm,marginparwidth=1.75cm]{geometry}

\usepackage{amsmath}
\usepackage{mathtools}
\usepackage{amsfonts}
\usepackage{xcolor}
\usepackage{overpic}
\usepackage{comment}
\usepackage{algpseudocode}
\usepackage{bbm}

\usepackage{graphicx} 
\usepackage{caption,subcaption}

\newcommand{\R}{\mathbb{R}}
\newcommand{\E}{\mathbb{E}}

\newcommand{\mtx}[1]{\mathbf{#1}}
\newcommand{\mycomment}[1]{}
\DeclareMathOperator{\Tr}{tr}
\DeclareMathOperator{\Hutch++}{Hutch\!+\!\!+}
\DeclareMathOperator{\ContHutch++}{ContHutch\!+\!+}

\DeclareMathOperator\erf{erf}
\DeclareMathOperator\sinc{sinc}

\title{contHutch++: Stochastic trace estimation for implicit integral operators\thanks{Funding: This work is supported by National Science Foundation grant No.~DMS-2045646. The work was also supported by the SciAI Center, and funded by the Office of Naval Research (ONR), under Grant Number N00014-23-1-2729.}}
\author{Jennifer Zvonek\thanks{Center for Applied Mathematics, Cornell University, Ithaca, NY 14853-4201, United States (\email{jez34@cornell.edu})}, \and
Andrew Horning\thanks{Department of Mathematics, MIT, Cambridge, MA 02142, United States (\email{horninga@mit.edu})}, \and Alex Townsend\thanks{Mathematics Department, Cornell University, Ithaca, NY 14853-4201, United States (\email{townsend@cornell.edu})}.}

\begin{document}
\maketitle

\begin{abstract}
Hutchinson's estimator is a randomized algorithm that computes an $\epsilon$-approximation to the trace of any positive semidefinite matrix using $\mathcal{O}(1/\epsilon^2)$ matrix-vector products. An improvement of Hutchinson's estimator, known as $\Hutch++$, only requires $\mathcal{O}(1/\epsilon)$ matrix-vector products. In this paper, we propose a generalization of $\Hutch++$, which we call $\ContHutch++$, that uses operator-function products to efficiently estimate the trace of any trace-class integral operator.
Our ContHutch++ estimates avoid spectral artifacts introduced by discretization and are accompanied by rigorous high-probability error bounds.  We use ContHutch++ to derive a new high-order accurate algorithm for quantum density-of-states and also show how it can estimate electromagnetic fields induced by incoherent sources. 
\end{abstract}



\section{Introduction}

Estimating the trace of a matrix from matrix-vector products is a foundational task in computational mathematics. This so-called \textit{matrix-free trace estimation} problem plays a crucial role in state-of-the-art techniques for log-determinants, eigenvalue counts, spectral densities, quantum correlations, and centrality measures on graphs~\cite{ubaru2017fast,cortinovis2021randomized,lin2016approximating,lin2017randomized,weisse2006kernel,Hutch++,PhysRevLett.107.114302,PhysRevB.92.134202,yao2022trace}. In a typical application, a large matrix can be multiplied with vectors efficiently, but the diagonal entries of the matrix cannot be computed explicitly~\cite{ubaru2017fast}. Instead, matrix-vector products with random probe vectors are used to estimate the trace with high probability. Structure in the matrix can often be leveraged to reduce the variance of these estimates~\cite{frommer2022multilevel,hallman2022multilevel,chen2023krylov,epperly2023xtrace}, and asymptotically optimal variance reduction strategies have been proposed and analyzed for symmetric positive-definite matrices~\cite{Hutch++,persson2022improved}.

In many applications, the matrix-free trace estimation problem is derived by discretizing a continuous problem. The trace of the resulting matrix approximates the trace of an infinite-dimensional operator. That is, given a trace-class operator $F:L^2(\Omega)\rightarrow L^2(\Omega)$ on domain $\Omega\subset \mathbb{R}^d$ with kernel $f:\Omega\times\Omega\rightarrow\mathbb{R}$, one wants to compute the trace\footnote{Every trace class operator can be associated with a Hilbert--Schmidt kernel that is integrable along its diagonal~\cite{brislawn1988kernels}.} (we use  ${\rm tr}(F)$ and ${\rm tr}(f)$ interchangeably throughout this paper):
\[
{\rm tr}(F) = \int_{\Omega} f(x,x)\,dx<\infty.
\]
The kernel $f(x,y)$ is usually not known explicitly but rather the operator-function product, defined by
\begin{equation}
u(x)\mapsto [Fu](x) = \int_{\Omega}f(x,y)\,u(y)\,dy, \qquad\text{for a.e.}\qquad x\in\Omega,
\label{eq:operatorFunction}
\end{equation}
can be evaluated with $u\in L^2(\Omega)$. For example, $f(x,y)$ might be a Green's function associated with a differential operator, and evaluating~\cref{eq:operatorFunction} involves solving a differential equation with right-hand side $u$ (see the examples in~\cref{sec:DOS} and~\cref{sec:field_intensity}).

In this paper, we propose a new randomized trace estimator that uses $m\geq 1$ operator-function products to approximate the trace of $F$ directly and achieves a specified error tolerance $\epsilon>0$ with probability $1-\delta>0$. Our estimator is \textit{discretization-oblivious}, meaning it does not depend on a particular discretization scheme to apply the operator $F$ to functions. Instead, our algorithms and analysis are consistent with any method for (approximately) evaluating the map $u\rightarrow Fu$ on a finite set of carefully constructed smooth random functions (see~\cref{sec:GP}) and inner products in $L^2(\Omega)$. By focusing on the operator, our method holds three significant advantages over methods that work directly with discretizations of the operator:
\begin{itemize}
    \item Our approximations converge to ${\rm tr}(F)$ in expectation, and rigorous error bounds with high probability guarantees are derived. We do not require uniform spectral approximation properties of the underlying discretizations, mitigating the impact of discretization-induced spectral artifacts such as non-converged eigenvalues, spectral pollution, and spectral invisibility.
    \item Our bounds reveal explicit relationships between the accuracy of the randomized estimator and the intrinsic properties of the kernel, $f(x,y)$, such as regularity (see~\cref{sec:ContHutch}).
    \item Any discretization method can be used for any trace-class kernel. Thus, optimal estimation rates for symmetric positive kernels apply even when sparse, nonsymmetric spectral methods are used for the fast and accurate solution of differential equations when evaluating $u\rightarrow Fu$.
\end{itemize}

Our approach to the operator-free trace estimator is based on a continuous analogue of the matrix-free trace estimator known as $\Hutch++$, which interpolates between Hutchinson's estimator and a randomized SVD estimator in a simple but asymptotically optimal way. \Cref{sec:background} introduces Hutchinson's estimator, the randomized range-finder, and $\Hutch++$ for matrices. In~\cref{sec:Results}, we construct a continuous analogue of Hutchinson's estimator and $\Hutch++$ and present key results on their accuracy and efficiency. Error bounds for the continuous analogue of Hutchinson's estimator and $\Hutch++$ are derived in~\cref{sec:ContHutch} and~\cref{sec:ContHutch++}, respectively. In~\cref{sec:num_exp}, we propose a novel algorithm that combines continuous trace estimation with high-order smoothing kernels for fast and accurate density-of-states calculations in quantum mechanics (see~\cref{sec:DOS}) and show how to apply our continuous trace estimation technique to a mean-field estimation problem inspired by photonic design (see~\cref{sec:field_intensity}).

\section{Background material}\label{sec:background}   
We begin by describing two algorithms for estimating the trace of a matrix: (1) Hutchinson's estimator (see~\cref{sec:Hutchinson}) and (2) $\Hutch++$ (see~\cref{sec:Hutch++}). $\Hutch++$ uses the randomized range finder (see~\cref{sec:randSVD}) to reduce the variance of Hutchinson's estimator. 

\subsection{Hutchinson's Estimator} \label{sec:Hutchinson}
Hutchinson's estimator uses matrix-vector products with random vectors to approximate the trace of a matrix when it is infeasible to calculate the trace directly~\cite{hutchinson}. Given a matrix $A \in \mathbb{R}^{n \times n}$, Hutchinson's estimate with $m$ probe vectors is given by 
\begin{equation*}
    H_m(A) = \frac{1}{m} \sum_{i=1}^m z_i^T A z_i,
\end{equation*}
where $z_i \in \mathbb{R}^n$ are independent, identically distributed (i.i.d.) random vectors~\cite{hutchinson}.  Common random distributions for $z_i$ include i.i.d.~standard Gaussian or Rademacher vectors, whose components are randomly chosen as $+1$ or $-1$. 
Since these random vectors have entries with mean 0 and variance 1, 
Hutchinson's estimator has the property that $\mathbb{E}[H_m(A)] = \Tr(A)$. 
Furthermore, in the case of standard Gaussian random vectors and symmetric positive semidefinite (PSD) matrices, we can explicitly determine how many matrix-vector products are needed to accurately estimate $\Tr(A)$ with high probability~\cite{avron2011randomized,roosta2015improved}. The lower bound on the number of matrix-vector products below indicates that Hutchinson's estimator improves when the eigenvalues of $A$ have slow decay~\cite[Theorem~3]{roosta2015improved}.
\begin{theorem} \label{thm:HutchinsonBound}
    Let $A \in \R^{n\times n}$ be symmetric PSD, 
    $0<\varepsilon<1$, and $0 < \delta < 1$.
    Let $z_1, \ldots, z_m$ be i.i.d.~standard Gaussian random vectors. Then, $|H_m(A) - \Tr(A)| < \varepsilon \Tr(A)$ with probability $ \geq 1-\delta$ if $m > 8 C_A \log(2/\delta) / \varepsilon^2$, where $C_A=\|A\|_2/\Tr(A)\leq 1$.
\end{theorem}
We formulate a continuous analogue of Hutchinson's estimator for the trace of a function $f:\Omega \times \Omega \to \R$ in~\cref{sec:Results} and prove an analogue of~\cref{thm:HutchinsonBound} in~\cref{sec:ContHutch} (see~\cref{cor:ContHutch2}). 

\subsection{Randomized Range Finder} \label{sec:randSVD}
A randomized range finder approximates the column space of a matrix $A\in\mathbb{R}^{n\times n}$ by constructing an approximate basis from matrix-vector products with random vectors (e.g. standard Gaussian vectors)~\cite{rSVD}.

\begin{algorithm}
    \caption{Randomized Range Finder} \label{rSVD_alg}
    \begin{algorithmic}[1]
        \Statex \textbf{Input:} $A \in \R^{n_1\times n_2}$, target rank $k$, and oversampling parameter $p$
        \Statex \textbf{Output:} Matrix with orthonormal columns $Q_k \in \R^{n_1 \times (k+p)}$, whose range approximates the range of $A$
     
        \State Sample a matrix $G \in \R^{n_2 \times (k+p)}$ with i.i.d.~entries drawn from a standard Gaussian distribution.
        \State Compute $Y = AG$.
        \State Compute the QR factorization $Y = Q_k R$.
    \end{algorithmic} 
    
\end{algorithm}
More precisely, given a target rank $k$ and a small oversampling parameter $p$, the randomized range finder uses $k+p$ matrix-vector products to approximate the span of the dominant $k$ singular vectors of $A$ with high probability (see~\cref{rSVD_alg})~\cite{rSVD}.

The randomized range finder offers another way to estimate ${\rm tr}(A)$. If $Q$ is the matrix with orthonormal columns computed with the randomized range finder in~\cref{rSVD_alg} and $A$ is well-approximated by a rank $k$ matrix, then $QQ^T A \approx A$. Therefore, $\Tr(A) \approx \Tr(QQ^T A) = \Tr(Q^T A Q)$ by the cyclic property of trace. The matrix $Q^T A Q$ is of size $(k+p) \times (k+p)$, which is typically much smaller than $A$ itself. This technique is particularly effective when the singular values of $A$ decay rapidly. 

\subsection{Hutch++} \label{sec:Hutch++}
Since Hutchinson's estimator and the randomized range finder are most effective in complimentary circumstances, they can be combined to obtain an improved estimate for the trace of a matrix.  $\Hutch++$ does this by splitting a symmetric PSD matrix $A$ into two terms: (1) $QQ^T A QQ^T$ and (2) $A - QQ^T A QQ^T$, where $Q$ is the orthonormal output of the randomized range finder in~\cref{rSVD_alg}~\cite{Hutch++}.  The trace of $QQ^T A QQ^T$ is computed exactly while the trace of $A - QQ^T A QQ^T$ is computed using Hutchinson's estimator. The result is an estimator that uses only $\mathcal{O}(1/\varepsilon)$ matrix-vector products to estimate $\Tr(A)$ to within a relative error bound of $\varepsilon$ (see~\cref{alg:Hutch++}).  
\begin{algorithm} 
    \caption{Hutch++} \label{alg:Hutch++}
    \begin{algorithmic}[1]
        \Statex \textbf{Input:} Matrix-vector products from a symmetric PSD matrix $A \in \R^{n\times n}$ and the number of queries $m$.
        \Statex \textbf{Output:} Estimate of $\Tr(A)$.
        
        \State Sample two random matrices, $S \in \R^{n \times m/3}$, $G \in \R^{n \times m/3}$, each with i.i.d.~Gaussian entries.
        \State Compute $Y = AS$.
        \State Compute the QR factorization $Y = QR$.
        \State \textbf{Return:} $\Hutch++(A) = \Tr(Q^T A Q) + \frac{3}{m} \Tr(G^T (I - QQ^T) A (I - QQ^T) G)$.
    \end{algorithmic}
\end{algorithm}
In particular, $\Hutch++$ has the following probabilistic error bound~\cite[Theorem 1.1]{Hutch++}: 
\begin{theorem} \label{thm:Hutch++}
    Let $\varepsilon>0$, $0<\delta <1$, and $A \in \mathbb{R}^{n\times n}$ be symmetric PSD.
    If $\Hutch++$ is used with $m = \mathcal{O}(\sqrt{\log(1/\delta)}/\varepsilon + \log(1/\delta))$ matrix-vector products, then with probability $\geq 1 - \delta$,
    \begin{equation*}
        | \Hutch++(A) - \Tr(A)| \leq \varepsilon \Tr(A).
    \end{equation*}
\end{theorem}
    
\Cref{sec:Results} extends the $\Hutch++$ estimator to one for continuous functions $f:\Omega \times \Omega \to \R$. Our algorithm follows the same general steps, but some added complications result from using functions instead of matrices. We derive a rigorous bound on the number of operator-function products required (see~\cref{cor:ContHutch++}) and argue in~\cref{sec:ContHutch++} that, in a suitable parameter regime, the continuous analogue achieves the same optimal scaling as the finite-dimensional result in~\cref{thm:Hutch++}.

\section{Trace Estimation for Functions} \label{sec:Results}
This section introduces the algorithms and corresponding error bounds for a continuous version of Hutchinson's estimator (see~\cref{sec:IntroContHutch}) and $\ContHutch++$ (see~\cref{intro_ContHutch++}). We defer the analysis of these estimators to~\cref{sec:ContHutch} and~\cref{sec:ContHutch++}, respectively, and focus on the algorithms and results here. 
To develop continuous versions of Hutchinson's estimator and $\Hutch++$, we must first establish continuous versions of the mechanics used for the discrete estimators. We describe Gaussian processes, which are the continuous analogues of random vectors, in~\cref{sec:GP} and quasimatrices, which are continuous analogues of matrices, in~\cref{sec:quasimatrix}. 

\subsection{Gaussian Processes} \label{sec:GP}
A function, $g$, drawn from a Gaussian process (GP) is an infinite dimensional analogue of a vector drawn from a multivariate Gaussian distribution in the sense that samples from $g$ follow a multivariate Gaussian distribution. More precisely, we write $g \sim \mathcal{GP}(0,K)$ for some continuous positive definite kernel $K:\Omega \times \Omega \to \R$ if for any $x_1, \ldots, x_n \in \Omega$,
$(g(x_1), \ldots, g(x_n))$ follows a multivariate Gaussian distribution with mean $(0, \ldots, 0)$ and covariance $K_{ij} = K(x_i, x_j)$ for $1 \leq i,j \leq n$.

We are particularly interested in the squared exponential covariance kernel 
\[
K_{\mathrm{SE}}(x,y) = \frac{1}{\ell \sqrt{2 \pi}}\exp\left( -(x-y)^2/(2\ell^2) \right),
\]
where $s_\ell = \ell \sqrt{2 \pi}$ is a scaling factor chosen such that $\int_\mathbb{R} K_{SE}(0,y) dy = \int_\mathbb{R} \frac{1}{s_\ell} e^{-y^2/(2 \ell ^2)} dy = 1$.
The length scale parameter $\ell$ determines how correlated samples of $g$ are. If $\ell$ is large, then the samples $g(x_1), \ldots, g(x_n)$ are highly correlated and $g$ is close to a constant function. If $\ell$ is small, then samples of $g$ are only weakly correlated and $g$ is usually a highly oscillatory function. In particular, functions $g \sim \mathcal{GP}(0,K_{SE})$ typically require more computational degrees of freedom to resolve as $\ell$ decreases. 

Notice that GPs produce continuous functions. Thus, if $g$ is drawn from $\mathcal{GP}(0,K)$, then there is a positive correlation between $g(x)$ and $g(y)$. This contrasts with the random vectors that are commonly used in Hutchinson's estimator and $\Hutch++$, which have independent entries. Thus, our estimates have extra quantities to account for this correlation, depending on $\ell$ and $\|K_{SE}\|_{op}$ (the largest singular value of $K_{SE}$). In the limit $\ell\rightarrow 0$, $K_{SE}$ approximates the identity so that $\|K_{SE}\|_{op} = 1 + o(\ell)$. Moreover, $\|K_{SE}\|_{op}\leq 1$ is implied by Young's convolution inequality. In statements below, one can view $\|K_{SE}\|_{op}$ as some modest constant for each $\ell$ that approaches 1 from below as $\ell\rightarrow0$.

The algorithms in this paper are compatible with any family of symmetric positive definite kernels approximating the idenity. However, their performance and approximation properties will depend on the kernel of choice. We focus on the squared exponential kernel to obtain explicit error bounds for the performance of our estimator, but state a few key results in terms of a generic GP with kernel $K$.

\subsection{Quasimatrices} \label{sec:quasimatrix}
Quasimatrices are an infinite dimensional analogue of tall-skinny matrices~\cite{townsend2015continuous}.  Let $\Omega_1,\Omega_2\subseteq\R$ be two domains and denote by $L^2(\Omega_{1})$ the space of square-integrable functions defined on $\Omega_{1}$.   Many of the results in this paper are easier to state using quasimatrices. 
We say that $\mtx{B}$ is an $\Omega_1\times k$ quasimatrix, if $\mtx{B}$ is a matrix with $k$ columns where each column is a function defined on $\Omega_1$. That is, 
\begin{equation*}
    \mtx{B} = \begin{bmatrix} b_1 \, | & \! \cdots \! & | \, b_k \end{bmatrix}, \qquad b_j:\Omega_1 \to \R.
\end{equation*}
Throughout this paper, we assume that the columns of $B$ are square-integrable (i.e., $b_j \in L^2(\Omega_1)$).

Quasimatrices are useful for defining analogues of matrix operations for Hilbert--Schmidt (HS) operators~\cite{de1991alternative,stewart1998matrix,townsend2015continuous,trefethen1997numerical}. 
For example, if $F:L^2(\Omega_1)\to L^2(\Omega_2)$ is a HS operator, then we write $F\mtx{B}$ to denote the quasimatrix obtained by applying $F$ to each column of $\mtx{B}$. 
Moreover, we write $\mtx{B}^*\mtx{B}$ and $\mtx{B}\mtx{B}^*$ to mean the following:
\begin{equation*}
    \mtx{B}^*\mtx{B} = \begin{bmatrix}\langle b_1,b_1 \rangle & \cdots & \langle b_1,b_k \rangle\\ \vdots & \ddots &  \vdots\\
    \langle b_k,b_1 \rangle & \cdots & \langle b_k,b_k \rangle \end{bmatrix}, \qquad \mtx{B}\mtx{B}^* = \sum_{j=1}^k b_j(x)b_j(y),
\end{equation*}
where $\langle \cdot, \cdot \rangle$ is the $L^2(\Omega_1)$ inner-product.  
A quasimatrix $Q$ has orthonormal columns if $Q^*Q$ is the identity matrix.
Many operations for rectangular matrices in linear algebra, such as the QR factorization, can be generalized to quasimatrices~\cite{townsend2015continuous}.

\subsection{Continuous Hutchinson's Estimator} \label{sec:IntroContHutch}
We now extend Hutchinson's estimator to trace estimation of continuous functions $f:\Omega \times \Omega \to \mathbb{R}$, where $\Omega = [a,b]$ is an interval in $\R$ with $-\infty < a<b < \infty$. That is, we want to approximate ${\rm tr}(f) = \int_\Omega f(x,x)dx$. In the discrete case, Hutchinson's estimator can use i.i.d.~standard Gaussian vectors to approximate ${\rm tr}(A)$. Thus, we propose using functions drawn from a GP to estimate the trace of a function. We propose the estimator
\begin{equation} \label{eq:hutch1}
    H_m(f) = \frac{1}{m} \sum_{i=1}^m \int_\Omega \int_\Omega g_i(x) f(x,y) g_i(y) dxdy,
\end{equation}
where each $g_i$ is independently drawn from $\mathcal{GP}(0,K)$. We can express the estimator in~\cref{eq:hutch1} in terms of quasimatrices. Let $F$ be the integral operator corresponding to the function $f \in L^2(\Omega \times \Omega)$ such that $(Fg)(x) = \int_\Omega f(x,y) g(y) dy$ for all $g \in L^2(\Omega)$ and $x \in \Omega$. Then, if $G$ is the $\Omega \times m$ quasimatrix whose $i$th column is $g_i$, we can rewrite the continuous Hutchinson's estimator as 
\begin{equation} \label{eq:cont_hutch_est}
    H_m(f) = \frac{1}{m} \sum_{i=1}^m [G^* (F G)]_{ii} = \frac{1}{m} \Tr(G^* (F G)),
\end{equation} 
where the operator-function product $FG$ must be computed before applying $G^*$ from the left.
Written in this form, we see that the continuous Hutchinson's estimator takes the same form as the discrete Hutchinson's estimator, except here, $FG$ requires integration instead of matrix-vector products.  

It can be shown that the continuous Hutchinson's estimator in~\cref{eq:cont_hutch_est} converges to $\Tr(f)$ as $\ell \to 0$. 
For a fixed $\ell>0$, note that $\mathbb{E}[H_m(f)]$ is not precisely $\Tr(f)$ because of the covariance kernel $K$ in the Gaussian process used to generate the functions $g_i$. 
However, as $\ell \to 0$, we find that $\E[H_m(f)] \to \Tr(f)$ (see~\cref{cor:ContHutch2}).
In particular, if $f$ is Lipschitz continuous and the squared exponential kernel is used with parameter $\ell = \mathcal{O}(\varepsilon/\sqrt{\log(1/\varepsilon)})$ then we have $| \E[H_m(f)] - \Tr(f)| < \varepsilon \|f\|_{L^2}$.
Furthermore, if $m = \mathcal{O}(\log(1/\delta) / \varepsilon^2)$ operator-function products are used, then $\mathbb{P} \left[ \left| H_m(f) - \mathrm{tr}(f) \right| \geq \varepsilon \|f\|_{L^2} \right] \leq \delta$ (see~\cref{thm:ContHutch1}).
For PSD functions, we have that $\|f\|_{L^2} \leq \Tr(f)$ so that we may write this bound in terms of $\Tr(f)$ to obtain a relative error bound.

\begin{figure}
    \centering
    \begin{minipage}{0.48\textwidth}
        \begin{overpic}[width=\textwidth]{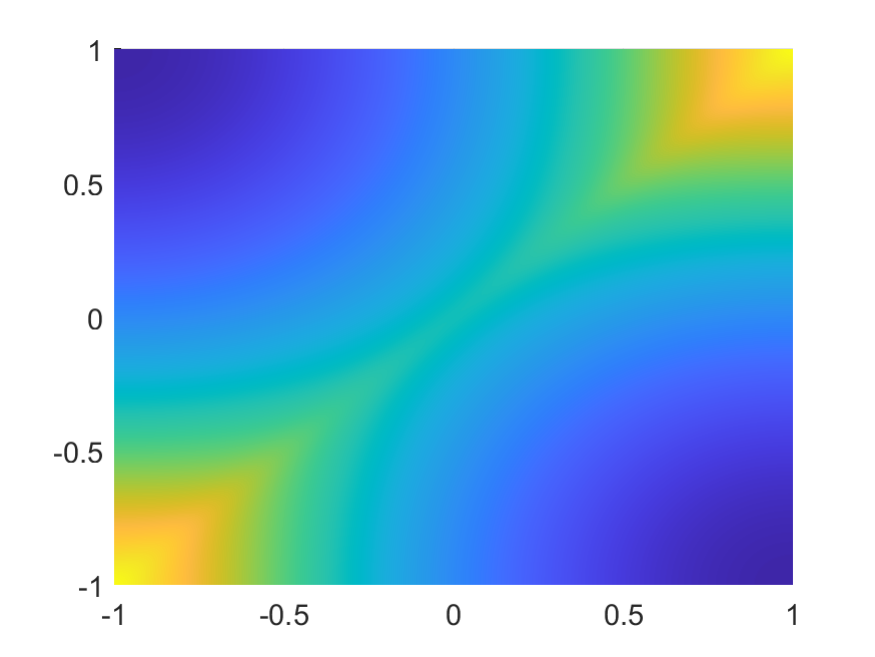}
        \put (50,-2) {$\displaystyle x$}
        \put (0,39) {$\displaystyle y$}
        \end{overpic}
    \end{minipage}
    \begin{minipage}{0.48\textwidth}
        \begin{overpic}[width=\textwidth]{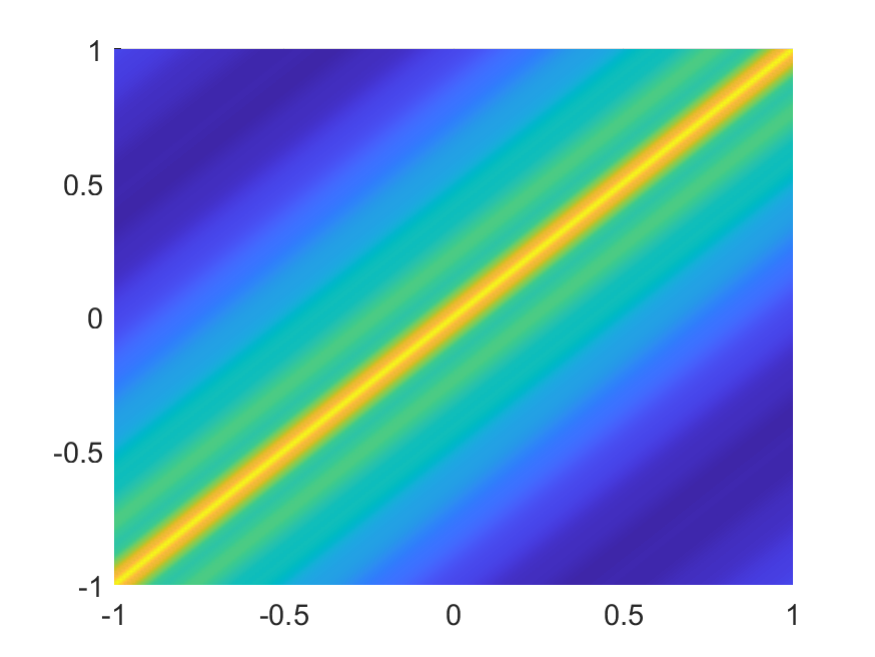}
        \put (50,-2) {$\displaystyle x$}
        \put (0,39) {$\displaystyle y$}
        \end{overpic}
    \end{minipage}
    \caption{\label{fig:toy_example1} Two symmetric positive definite kernels are used for trace estimation experiments: the Helmholtz-like kernel in~\cref{eqn:kernel1} (left panel) and the combination of sinc kernels in~\cref{eqn:kernel2} (right panel).}
\end{figure}

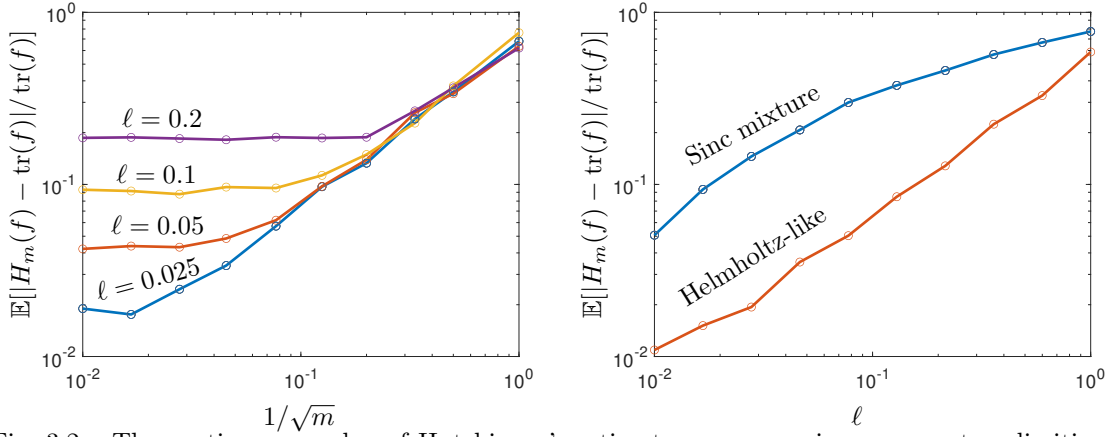
\begin{figure}
    \centering
    \begin{minipage}{0.48\textwidth}
        \begin{overpic}[width=\textwidth]{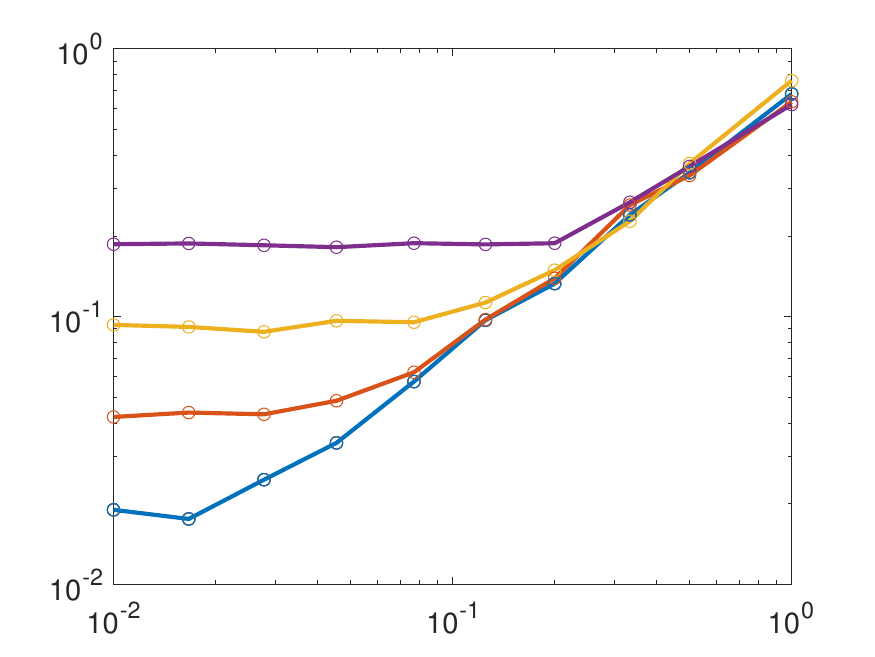}
        \put (45,-4) {$1/\sqrt{m}$}
        \put (0,15) {\rotatebox{90} {$\displaystyle \mathbb{E}[|H_m(f) - \Tr(f)|/\Tr(f)]$}}
        \put (20,49) {$\ell = 0.2$}
        \put (19,39) {$\ell = 0.1$}
        \put (18,30) {$\ell = 0.05$}
        \put (15,18) {\rotatebox{15} {$\ell = 0.025$}}
        \end{overpic}
    \end{minipage}
    \begin{minipage}{0.48\textwidth}
        \begin{overpic}[width=\textwidth]{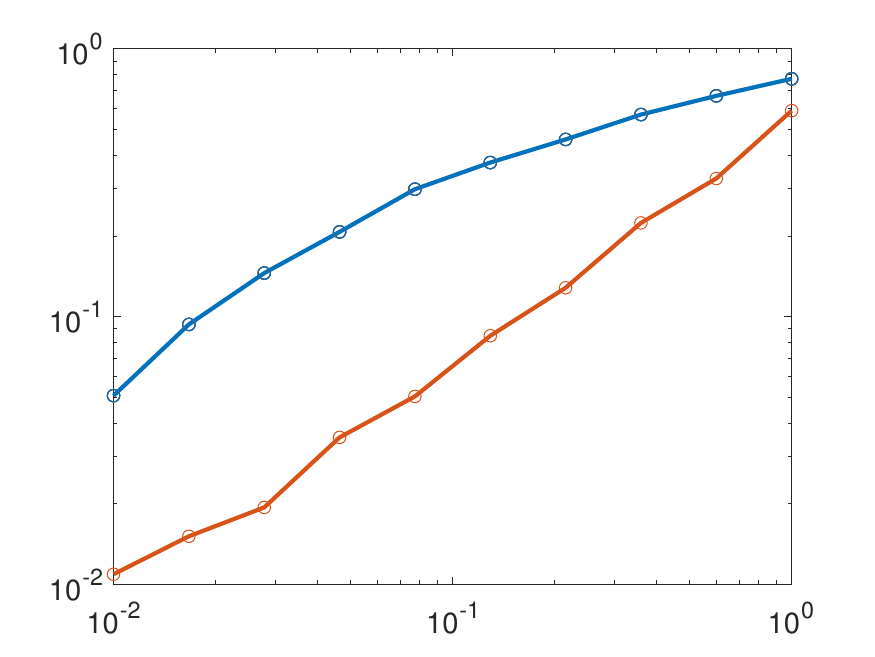}
        \put (48,-4) {$\displaystyle \ell$}
        \put (0,15) {\rotatebox{90} {$\displaystyle \mathbb{E}[|H_m(f) - \Tr(f)|/\Tr(f)]$}}
        \put (18,42) {\rotatebox{28} {Sinc mixture}}
        \put (17,17) {\rotatebox{32} {Helmholtz-like}}
        \end{overpic}
    \end{minipage}
    \caption{\label{fig:toy_example1_conv} The continuous analog of Hutchinson's estimator converges in mean up to a limiting bias determined by the covariance parameter $\ell>0$ for the Gaussian process, as demonstrated for the Helmholtz-like kernel in~\cref{eqn:kernel1} (left panel). This limiting bias decreases at a controlled rate as $\ell\rightarrow 0$ when $m$ is selected adaptively to balance the sample error and covariance error (right panel). The precise rate of decrease depends on the kernel's modulus of continuity (see~\cref{thm:expectation1}).}
\end{figure}

To demonstrate the convergence of the continuous Hutchinson's estimator, we tested the estimator on two functions: a Helmholtz-like function (left panel of~\cref{fig:toy_example1})
\begin{align}\label{eqn:kernel1}
    f(x,y) &= \left(1 - \cos \left(\frac{\pi(x+1)}{4} \right) \right) \sin \left(\frac{\pi(y+1)}{4} \right)\\ &+ \frac{1}{1 + e^{5(x-y)}} \left(1 - \cos \left(\frac{\pi(y+1)}{4} \right) \sin \left(\frac{\pi (x+1)}{4} \right) \right), \nonumber 
\end{align}
and a combination of three sinc functions (right panel of~\cref{fig:toy_example1})
\begin{equation}\label{eqn:kernel2}
    f(x,y) = \sinc(x-y) + \frac{1}{2} \sinc(10(x-y)) + \frac{1}{4} \sinc(50(x-y)),
\end{equation}
 both shown in~\cref{fig:toy_example1}. The convergence results are shown in~\cref{fig:toy_example1_conv}. In the left panel, we see that the relative accuracy of the estimator for the trace of the Helmholtz-like function converges linearly in expectation with $1/\sqrt{m}$ up to some bias determined by the value of $\ell$. This bias decreases as $\ell$ decreases. A similar trend holds for the combination of sinc functions. The right panel shows the relative accuracy of the estimator for both test functions. We see in the right panel that the trace estimate for both functions converges to the exact value as $\ell$ decreases. The exact convergence rate depends on the function's modulus of continuity (see~\cref{thm:expectation1} for more details).

\subsection{ContHutch++} \label{intro_ContHutch++}
Just as $\Hutch++$ uses the randomized range finder to improve on Hutchinson's estimator, we can use the continuous randomized range finder~\cite{ContRSVD} to improve on the continuous Hutchinson's estimator for estimating the trace of a symmetric PSD function. 
Note that a function $f:\Omega \times \Omega \to \R$ is PSD if for any $x_1, \ldots, x_n \in \Omega$, the matrix $A \in \R^{n \times n}$ with 
$A_{ij} = f(x_i, x_j)$ is PSD. 
To do so, we split the integral operator $F$ corresponding to the function $f$ into two terms: (1) $QQ^* F QQ^*$ and (2) $F - QQ^* F QQ^*$, where $Q$ is the quasimatrix computed by the continuous randomized range finder applied to $F$.
The trace of the first term can be rewritten as the trace of a matrix and computed exactly as $\Tr(QQ^* F QQ^*) = \Tr(Q^* F Q)$.
The trace of the second term is approximated using the continuous Hutchinson's estimator (see~\cref{alg:ContHutch++}). 
Note that if $F$ is the integral operator with continuous, symmetric PSD kernel $f$, then $\Tr(F) = \Tr(f)$~\cite{brislawn1988kernels}.

\begin{algorithm}[t]
    \caption{\label{alg:ContHutch++} ContHutch++}
    \begin{algorithmic}[1]
        \Statex \textbf{Input:} Operator-function products with the integral operator $F$ associated with a continuous, symmetric, PSD function $f:\Omega \times \Omega \to \mathbb{R}$, number $m$ of queries, and length-scale parameter $\ell$.
        \Statex \textbf{Output:} Estimate of $\Tr(f)$.

        \State Sample two random $\Omega \times m/3$ quasimatrices, $S$ and $G$, whose i.i.d.~columns are each drawn from $\mathcal{GP}(0, K_{SE})$. 
        \State Compute $Y = FS$. 
        \State Compute the QR factorization $Y=QR$ to find an $\Omega \times m/3$ quasimatrix, $Q$, with orthonormal columns.
        \State \textbf{Return:} $\ContHutch++(f) = \Tr(Q^T F Q) + \frac{3}{m}\Tr(\tilde{G}^T (F \tilde{G}))$, where $\Tilde{G} = (I - QQ^T)G$. 
    \end{algorithmic}
\end{algorithm}
To illustrate the improved performance of $\ContHutch++$, we test it on the two example functions from~\cref{sec:IntroContHutch}. The spectral decay of both functions is shown in the left panel of~\cref{fig:toy_example2}. The singular values of the Helmholtz-like function initially decay much faster than the singular values of the sinc mixture function. $\ContHutch++$ converges at a slow but steady rate during the spectral plateaus but can also exploit spectral decay to achieve high accuracy (see~\cref{cor:ContHutch++}).

In analogy with the matrix case,~\cref{fig:toy_example2} illustrates that $\ContHutch++$ may be spectrally accurate. That is, $\ContHutch++$ may be much more accurate than the worst case sample rate in~\cref{cor:ContHutch++} indicates when the singular values of $f$ decay rapidly. This is due to a combination of three key facts: as $m$ increases, (1) $\Tr(Q^*FQ)$ is a spectrally accurate approximation to $\Tr(F)$, (2) $H_m(F-QQ^*FQQ^*)$ concentrates around its mean, and (3) continuous Hutchinson's bias is controlled relative to $\|F-QQ^*FQQ^*\|_{L^2}$ by the length parameter $\ell$. While (1) and (2) are similar to statements for $\Hutch++$, the bias due to $g\sim GP(0,K_{SE})$ makes (3) unique to the continuous setting and depends on the regularity of $f$ and its low rank approximations (see the discussion after~\cref{thm:expectation1}).

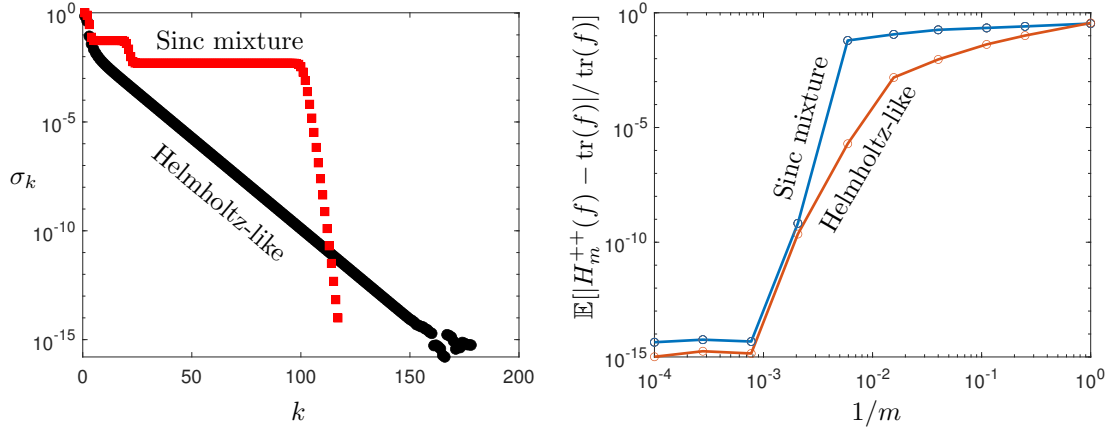
\begin{figure}
    \centering
    \begin{minipage}{0.48\textwidth}
        \begin{overpic}[width=\textwidth]{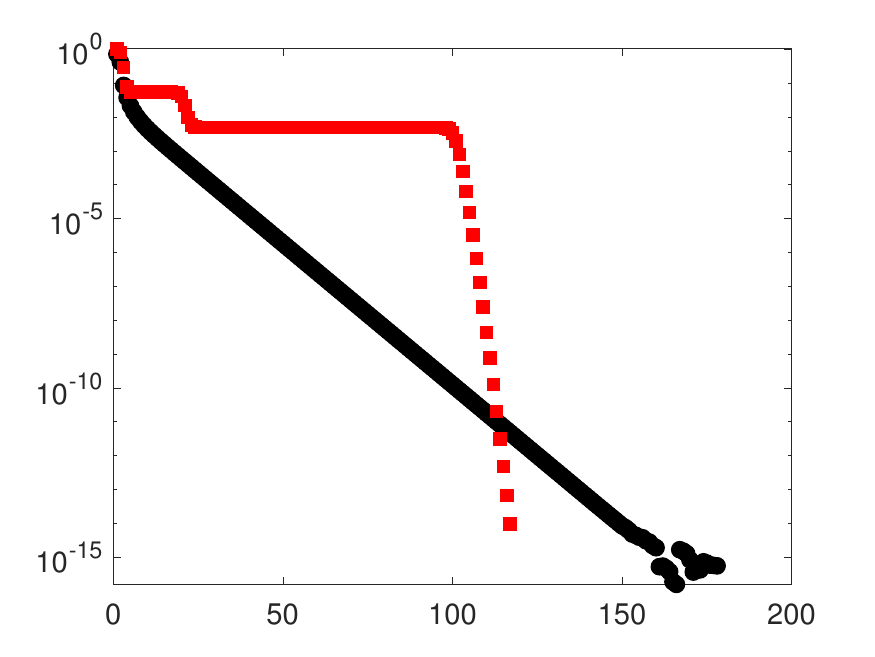}
        \put (50,-3) {$\displaystyle k$}
        \put (0,39) {$\displaystyle \sigma_k$}
        \put (26,63) {Sinc mixture}
        \put (25,44) {\rotatebox{-40} {Helmholtz-like}}
        \end{overpic}
    \end{minipage}
    \begin{minipage}{0.48\textwidth}
        \begin{overpic}[width=\textwidth]{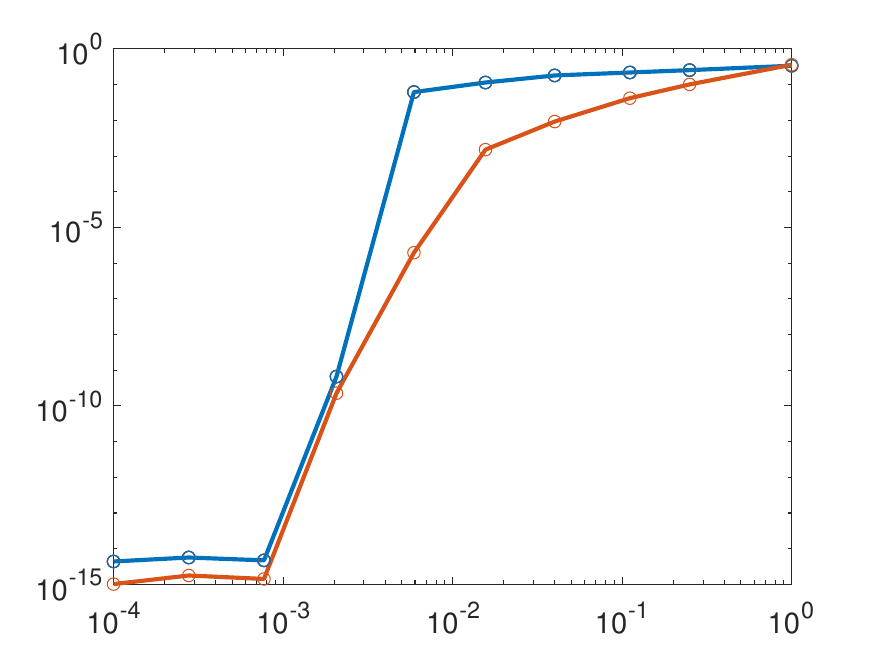}
        \put (48,-3) {$\displaystyle 1/m$}
        \put (-2,15) {\rotatebox{90} {$\displaystyle 
        \mathbb{E}[|H_m^{++}(f) - \Tr(f)|/\Tr(f)]$}}
        \put (34,36) {\rotatebox{75} {Sinc mixture}}
        \put (42,30) {\rotatebox{60} {Helmholtz-like}}
        \end{overpic}
    \end{minipage}
    \caption{\label{fig:toy_example2} The $\Hutch++$ estimator leverages spectral decay to compute ${\rm tr}(f)$ with high-accuracy. The spectrum of the two kernels is shown in the left panel, and the expected relative error in the trace (estimated from samples) is plotted against $1/m$ in the right panel (with $\ell=0.05$).}
\end{figure}

\section{Analyzing the Continuous Hutchinson's estimator} \label{sec:ContHutch}
We now turn to analyzing the continuous Hutchinson's estimator. \Cref{thm:expectation1} establishes that the length-scale parameter $\ell$ determines the bias of $H_m(f)$, i.e., how close the expected value of our estimate is to the correct trace. In particular, it shows that $\ell$ should scale with $\epsilon$ or the modulus of continuity of $f$, whichever is smaller, to achieve a bias of $\epsilon\|f\|_{L^2}$. For example, if $f$ is Lipschitz continuous in the second variable, then $d = \mathcal{O}(\varepsilon)$ as $\varepsilon\rightarrow 0$ and we take $\ell=\mathcal{O}(\varepsilon/\sqrt{\log(1/\varepsilon)})$ to achieve $|\mathbb{E}[H_m(f) - \Tr(f)| < \varepsilon \|f\|_{L^2}$. Continuous functions with less regularity, such as H\"{o}lder continuous families, may typically require smaller $\ell$ to achieve comparable accuracy. We denote the max norm by $\|f\|_\infty = \max_{x,y \in \Omega} |f(x,y)|$.
\begin{theorem} \label{thm:expectation1}
    Let $\varepsilon>0$ and $f:\Omega \times \Omega \to \mathbb{R}$ be a continuous function, where $\Omega = [a,b]$ with $-\infty < a<b < \infty$. Let $d = d(\varepsilon)$ be such that $|f(x,y) - f(x,z)| \leq \frac{\varepsilon \|f\|_{L^2}}{4(b-a)}$ for all $y,z \in \Omega$ such that $|y-z| < d$. If $H_m(f)$ is the continuous Hutchinson's estimate of $\Tr(f)$ with kernel $K_{SE}$ and length-scale parameter
    $\ell <  \min \left \{ \frac{d}{\sqrt{2 \log(8 \|f\|_{\infty} (b-a) / (\varepsilon \|f\|_{L^2}))}}, \frac{5 \varepsilon \|f\|_{L^2}}{26 \|f\|_\infty} \right \}$, then
$|\mathbb{E}[H_m(f)] - \Tr(f)| < \varepsilon \|f\|_{L^2}$.
\end{theorem}
\begin{proof}
The expectation of $H_1(f)$ is
    \begin{align*}
        \mathbb{E}[H_1(f)] &= \int_\Omega \int_\Omega f(x,y) \mathbb{E} [g_i(x) g_i(y)] dxdy = \int_\Omega \int_\Omega f(x,y) K_{SE}(x,y) dxdy,
    \end{align*} 
    since $K_{SE}(x,y) = \mathbb{E} \left[(g_i(x) - \mathbb{E}[g_i(x)]) (g_i(y) - \mathbb{E}[g_i(y)]) \right] = \mathbb{E}[g_i(x) g_i(y)]$.
    Thus, we find that
    \begin{align*}
        \mathbb{E}[H_m(f)] = \frac{1}{m} \sum_{i=1}^m \mathbb{E}[H_1(f)] = \mathbb{E}[H_1(f)] = \int_\Omega \int_\Omega f(x,y) \frac{1}{\ell \sqrt{2 \pi}} \exp (-(x-y)^2/(2\ell ^2)) dxdy.
    \end{align*}
Since $\int_\R K_{SE}(x,y) dy = 1$, we can write $f(x,x) = \int_\R f(x,x) K_{SE}(x,y) dy$.
    This means that we have,
    \begin{align}
        \left| \E[H_m(f)] - 
        \Tr(f) \right| 
        &\leq \int_\Omega \left| \int_\Omega f(x,y) K_{SE}(x,y) dy - f(x,x) \right| dx \nonumber \\
        &\leq \underbrace{\int_\Omega \left| \int_\Omega (f(x,y) - f(x,x))K_{SE}(x,y) dy \right| dx}_{I_1}
        + \underbrace{\int_\Omega \left| \int_{\R \setminus \Omega} f(x,x) K_{SE}(x,y) dy \right| dx}_{I_2}. \label{eq:bound2}
    \end{align}
We will show that $I_1 \leq \varepsilon \|f\|_{L^2}/2$ and $I_2 \leq \varepsilon \|f\|_{L^2}/2$ so that $|\E[H_m(f)] - \Tr(f)| \leq \varepsilon \|f\|_{L^2}$.
To evaluate $I_1$ in~\cref{eq:bound2}, we partition the domain of integration $\Omega \times \Omega$ into $\Omega \times \{y: |x-y|<d\}$ and $\Omega \times \{y: |x-y| \geq d\}$. We then use the fact that $|f(x,y) - f(x,x)| < \varepsilon \|f\|_{L^2}/(4(b-a))$ when $|x-y|<d$ to simplify the inner integral in $I_1$ to
    \begin{align*}
        I_1 &\leq \frac{1}{\ell \sqrt{2 \pi}} \int_\Omega \int_{|x-y|< d}  \frac{\varepsilon \|f\|_{L^2}}{4(b-a)} \exp (-(x-y)^2/(2\ell^2)) dy dx \\
        & \quad + \frac{1}{\ell \sqrt{2 \pi}} \int_\Omega \int_{|x-y|\geq d}  |f(x,y) - f(x,x)| \exp (-(x-y)^2/(2\ell^2)) dy dx.
    \end{align*}
    Since $\int_{|x-y|< d}  K_{SE}(x,y) dy \leq \int_\mathbb{R}  K_{SE}(x,y) dy = 1$ and $|f(x,y) - f(x,x)| \leq |f(x,y)| + |f(x,x)| \leq 2 \|f\|_\infty$, we have 
    \begin{align*}
        I_1 &\leq \int_\Omega \frac{\varepsilon \|f\|_{L^2}}{4(b-a)} dx + \frac{2 \|f\|_\infty}{\ell \sqrt{2 \pi}} \int_\Omega \int_{|x-y|\geq d} \exp (-(x-y)^2/(2\ell^2)) dy dx.
    \end{align*}
    By evaluating the first term exactly and bounding the second term, we find that 
    \begin{align*}
        I_1 &\leq \frac{\varepsilon \|f\|_{L^2}}{4} + 2 \|f\|_{\infty} (b-a) \left( 1 - \erf\left(\frac{d}{\ell \sqrt{2}} \right) \right) \\
        & \leq \frac{\varepsilon \|f\|_{L^2}}{4} + 2 \|f\|_{\infty} (b-a) \exp \left( -\frac{d^2}{2\ell^2} \right),
    \end{align*}
    where $\erf$ denotes the error function, and the second inequality comes from the fact that $1 - \erf(x) \leq \exp(-x^2)$ for $x>0$. Thus, we have that $I_1 \leq \varepsilon \|f\|_{L^2}/2$ since $\ell < d / \sqrt{2 \log(8 \|f\|_{\infty} (b-a) / (\varepsilon \|f\|_{L^2}))}$.
    
    To bound $I_2$ in~\cref{eq:bound2}, we partition the domain of integration $\Omega \times (\R \setminus\Omega)$ into three pieces $D_1$, $D_2$, and $D_3$ such that $\Omega \times (\R \setminus \Omega) = D_1 \cup D_2 \cup D_3$: 
    \begin{align*}
        D_1 &= \left([a, a + \ell \sqrt{2}] \times [a - \ell \sqrt{2}, a] \right)
        \cup \left( [b - \ell \sqrt{2}, b] \times [b, b + \ell \sqrt{2}] \right),\\
        D_2 &= \left((a + \ell \sqrt{2}, b] \times [a - \ell \sqrt{2}, a] \right)
        \cup \left( [a, b - \ell \sqrt{2}) \times [b, b + \ell \sqrt{2}] \right),\\
        D_3 &= \left( [a,b] \times (-\infty, a - \ell \sqrt{2}) \right) 
        \cup \left( [a,b] \times (b + \ell \sqrt{2}, \infty) \right).
    \end{align*}
    See~\cref{fig:Domains} for a depiction of these three domains. 
    \begin{figure}
        \centering
        \includegraphics[width = \linewidth]{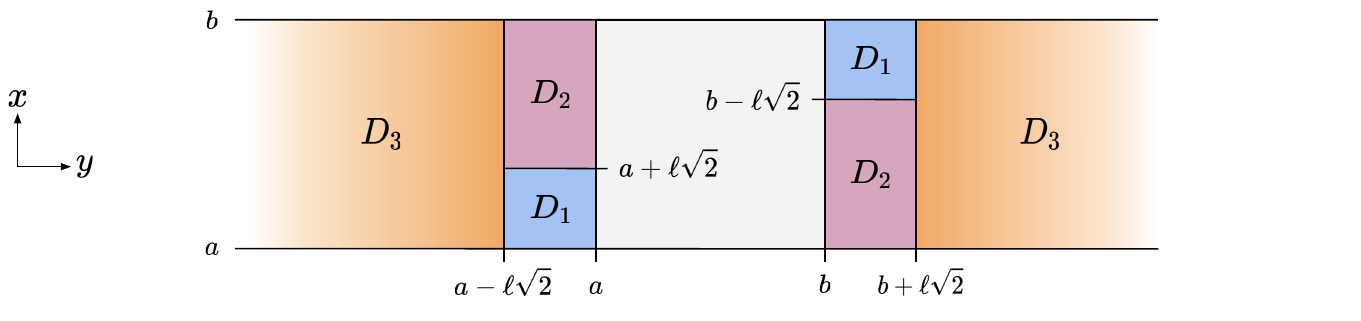}
        \caption{Domains of integration for computing the second term in~\cref{eq:bound2}.}
        \label{fig:Domains}
    \end{figure}
    On $D_1$, we use the fact that $f(x,x) \leq \|f\|_\infty$ for all $x \in \Omega$ and that $K_{SE}(x,y) \leq 1/\ell \sqrt{2\pi}$ to obtain a bound of $2\ell \sqrt{2/\pi} \|f\|_\infty$. In domains $D_2$ and $D_3$, we know that $|x-y|/(\ell \sqrt{2}) \geq 1$, so we have that $\exp\left(-(x-y)^2/(2\ell^2)\right) \leq \exp \left( - |x-y|/(\ell \sqrt{2}) \right)$. Using this bound, we can evaluate the integral over $D_2$ directly to get a bound of $2 \sqrt{2/\pi} \ell \|f\|_\infty ( e^{-1} - e^{-2} - e^{-(b-a)/(\ell \sqrt{2})} + e^{-(b-a + \ell \sqrt{2})/(\ell\sqrt{2})})$. Similarly, bounding the integral over $D_3$ gives a bound of $2 \sqrt{2/\pi} \ell \|f\|_\infty (e^{-1} - e^{-(b-a + \ell \sqrt{2})/(\ell \sqrt{2})})$. Combining our three error bounds, we find that 
    \begin{align*}
        \int_\Omega \left| \int_{\R \setminus \Omega} f(x,x) K_{SE}(x,y) dy \right| dx &\leq 2 \sqrt{\frac{2}{\pi}} \left( 1 + 2e^{-1} - e^{-2} - e^{-(b-a)/(\ell \sqrt{2})} \right) \ell \|f\|_\infty \\
        &\leq 2 \sqrt{\frac{2}{\pi}} \left( 1 + 2e^{-1} - e^{-2}\right) \ell \|f\|_\infty \\
        &\leq \frac{13}{5} \ell \|f\|_\infty \leq \varepsilon \|f\|_{L^2} / 2
    \end{align*}
    where the last inequality follows from $\ell < 5\varepsilon \|f\|_{L^2}/(26\|f\|_\infty)$. 
\end{proof}

When the continuous Hutchinson's estimator is used in conjunction with the randomized SVD in $\ContHutch++$ (c.f.~\cref{sec:ContHutch++}), the $L^2$ norm of $f$ in~\cref{thm:expectation1} may be very small. For Lipschitz continuous functions with constant $\alpha>0$ as above, \Cref{thm:expectation1} illustrates that $\ell$ should scale with $\epsilon \min\{\|f\|_{L^2}/\alpha,\|f\|_{L^2}/\|f\|_{\infty}\}$ to maintain an accuracy on the order of $\epsilon \|f\|_{L^2}$. When $f$ is reasonably smooth, then $\|f\|_{L^2}$ is typically comparable to $\|f\|_\infty$ and $\alpha$. Therefore, the bias of Hutchinson's estimator may remain on the order of $\|f\|_{L^2}$ over a large range of length scales. We believe this explains the phenomenon of spectral accuracy of $\ContHutch++$ observed in~\cref{fig:toy_example2} with $\ell=0.05$.

\subsection{Hanson--Wright inequality} 
The Hanson--Wright inequality~\cite{Ethan_HW_bound} is a result in probability theory that gives an upper bound on the tail probability of a quadratic form of independent, mean-zero subgaussian random vectors. Specifically, it states that if $x \in \mathbb{R}^n$ is a mean-zero, subgaussian random variable with cumulant generating function $\xi_x(t)$ satisfying $\xi_x(t) \leq \frac{1}{2} \nu t^2$ for some $\nu > 0$, and $\epsilon > 0$, then for any $t > 0$,
\begin{equation} \label{eq:discrete_HW}
    \mathbb{P} \left( \left|x^T A x - \mathbb{E}\left[x^T A x \right] \right|  \geq t \right) \leq \exp \left( - \frac{t^2}{80 \nu^2 \|A\|_F^2 + 16 \nu \|A\| t} \right).
\end{equation}

The inequality is beneficial for analyzing the behavior of random matrices and is often used in high-dimensional statistics and machine learning. In the context of Hutchinson's estimator, it is used to show that Hutchinson's estimate is accurate with exceptionally high probability~\cite{Hutch++}. We present a continuous version of the Hanson--Wright inequality to obtain a similar high-probability result. 

The following theorem tells us that if $f$ is a symmetric PSD function, we can bound the probability that the continuous Hutchinson's estimate is far from its expected value. To do so, we define the operator $\phi: L^2(\Omega) \to \R$ in terms of $f$ such that $\phi(g)$ is the continuous Hutchinson's estimate of $\Tr(f)$ with $m=1$. We first show that $\phi(g)$ is close to its expected value with high probability, then in~\cref{sec:ContHutch_bound} we use this result to show that the full continuous Hutchinson's estimate is close to $\Tr(f)$ with high probability.

\begin{theorem} \label{thm:HW}
Let $g\sim\mathcal{GP}(0, K)$ for some continuous, PSD kernel $K:\Omega \times \Omega \to \mathbb{R}$ and $f:\Omega \times \Omega \to \mathbb{R}$ be a symmetric, PSD, continuous function, where $\Omega = [a,b]$ with $-\infty< a<b < \infty$. Define $\phi(g) = \int_\Omega \int_\Omega g(x) f(x,y) g(y) dxdy$. Then, we have
    \[
    \begin{aligned}
        \mathbb{P} \left( \left|\phi(g) - \mathbb{E}[\phi(g)]  \right| \geq t \right) & \leq 2 \exp \left(- \frac{t^2}{80 \|f\|_{L^2}^2 \|K\|_{op}^2 + 16 \|f\|_{op} \|K\|_{op} t}  \right), \qquad t>0.
\end{aligned}\]
\end{theorem}
\begin{proof}
    See~\cref{sec:HW_appendix}.
\end{proof}

A similar result to~\cref{thm:HW} holds for all continuous functions that can be decomposed as $f = f_1 - f_2$, where $f_1, f_2$ are continuous PSD functions in $L^2(\Omega \times \Omega)$, with the same assumptions as in~\cref{thm:HW}.

\begin{corollary} \label{cor:HW}
Let $f:\Omega \times \Omega \to \R$ be any continuous function that can be written as $f = f_1 - f_2$, for continuous PSD functions $f_1, f_2$ in $L^2(\Omega \times \Omega)$. Then, 
\[
     \mathbb{P}\left( |\phi(g) - \E[\phi(g)]| \geq t  \right) \leq 4 \exp \left(- \frac{t^2}{ 80\|f\|_{L^2}^2 \|K\|_{op}^2 + 16\|f\|_{op} \|K\|_{op} t }  \right).
\]
\end{corollary}
\begin{proof}
    See~\cref{sec:HW_appendix}.
\end{proof}

In many cases, $f$ can be decomposed as $f = f_1 - f_2$ as required for~\cref{cor:HW}.
Note that defining $\phi$ in terms of a continuous function $f$ is the same as defining $\phi$ in terms of the symmetric function $\Tilde{f}(x,y) = \frac{1}{2} (f(x,y) + f(y,x))$. Thus, we can treat $f$ as symmetric for the following argument.
 When $f : \Omega \times \Omega \to \mathbb{R}$ is continuous and symmetric, the integral operator $F$ defined by $(Fg)(x) = \int_\Omega f(x,y)g(y)dy$ for $x \in \Omega$ is self-adjoint and compact. This means we can write down an orthonormal basis of eigenfunctions $e_1, e_2, \ldots $ and eigenvalues $\lambda_1, \lambda_2, \ldots$ such that $Fe_j = \lambda_j e_j$, where equality holds in the $L^2$ sense.  For eigenfunctions corresponding to nonzero eigenvalues, $e_j$ is a continuous function (by the same reasoning as in the proof of Mercer's Theorem). If there are a finite number of positive eigenvalues or a finite number of negative eigenvalues, then we can write $f$ as $f = f_1 - f_2$, where equality holds in the $L^2$ sense. 
Let $S = \{i : \lambda_i \geq 0\}$. 
Then $f_1$ and $f_2$ are given by $f_1(x,y) = \sum_{i \in S} \lambda_i e_i(x) e_i(y)$ and $f_2(x,y) = -\sum_{j \in \mathbb{N}\setminus S} \lambda_j e_j(x) e_j(y)$, which gives two PSD functions. 

If $F$ has an infinite number of positive and negative eigenvalues, we cannot choose such a decomposition for $f_1$ and $f_2$. However, if $f$ is Lipschitz continuous, then the eigendecomposition of $f$ is absolutely convergent~\cite{smithies} and we can decompose $f$ into two PSD functions as we did for the case of finitely many positive or negative eigenvalues. In practice, there are many Green's functions that do not have an infinite number of both positive and negative eigenvalues. For example, Green's functions corresponding to uniformly self-adjoint elliptic partial differential equations have infinitely many positive eigenvalues but, at most, a finite number of negative eigenvalues. 

\subsection{An Error Bound for the Continuous Hutchinson's Estimator} \label{sec:ContHutch_bound}
We now use our continuous Hanson--Wright inequality to prove that the continuous Hutchinson's estimator is accurate with exceptionally high probability. 
\begin{theorem} \label{thm:ContHutch1}
   Let $0<\delta<1$ and $f: \Omega \times \Omega \to \mathbb{R}$ be a continuous function with $\Omega = [a, b]$ and $-\infty < a < b < \infty$, where $f$ can be written as $f = f_1 - f_2$ for some continuous PSD functions $f_1, f_2: \Omega \times \Omega \to \R$.
    Let $m \geq 96 \log(4/\delta)$.
    Then, with probability $\geq 1-\delta$, 
\[
\left| H_m(f) - \mathbb{E}[H_m(f)] \right| \leq \|K\|_{op} \|f\|_{L^2} \sqrt{\frac{96 \log(4/\delta)}{m}},
\]
    where $H_m(f)$ is the continuous Hutchinson's estimate of $\mathrm{tr}(f)$ with $m$ probe functions drawn from $\mathcal{GP}(0,K)$ for some continuous PSD covariance kernel $K:\Omega \times \Omega \to \mathbb{R}$. 
\end{theorem}

\begin{proof}
We first define new functions, $\Bar{f}$ and $\Bar{g}$, defined by block repetitions of $f$ and $g$, respectively. This will allow us to use the continuous Hanson--Wright inequality (see~\cref{cor:HW}) to bound the failure probability of the continuous Hutchinson's estimate of $\Tr(f)$ with $m$ probe functions. To do so, define the intervals $\Omega_i = [a + i(b-a), a + (i+1)(b-a))$ for $0\leq i \leq m-2$ and $\Omega_{m-1} = [a + (m-1)(b-a), a + m(b-a)]$. Define the union of these domains as $\Bar{\Omega} = \bigcup_{i=0}^{m-1} \Omega_i = [a, m(b - a) + a]$. Then, since the $\Omega_i$'s are disjoint, for any $x \in \Bar{\Omega}$, there exists a unique $0 \leq i \leq m-1$ such that $x \in \Omega_i$. Thus, we define our block extension of $f$ to be $\Bar{f}: \Bar{\Omega} \times \Bar{\Omega} \to \mathbb{R}$ such that
    \begin{equation*}
        \Bar{f}(x,y) = \begin{cases}
            f \left(x - i(b-a), y - i(b-a) \right), & \text{if } x, y \in \Omega_i, \\
            0, & \text{otherwise.}
        \end{cases}
    \end{equation*}
    Similarly, if $g_0, \ldots, g_{m-1}$ are i.i.d.~functions drawn from $\mathcal{GP}(0, K_{SE})$, we define $\Bar{g}:\Bar{\Omega} \to \mathbb{R}$ by
    \begin{equation*}
        \Bar{g}(x) = g_i(x - i(b-a)), \qquad x\in \Omega_i.
    \end{equation*}
    Note that $\Bar{g}$ then has the covariance kernel $\Bar{K}:\Bar{\Omega} \times \Bar{\Omega} \to \mathbb{R}$, which is defined as a block extension of the covariance kernel $K$ in the same way that $\Bar{f}$ is defined in terms of $f$. 

We can then rewrite the continuous Hutchinson's estimator in terms of $\Bar{f}$ and $\Bar{g}$ so that
    \begin{equation*}
        \frac{1}{m} \sum_{i=1}^m \int_\Omega \int_\Omega g_i(x) f(x,y) g_i(y) dx dy = \frac{1}{m} \int_{\Bar{\Omega}}\int_{\Bar{\Omega}} \Bar{g}(x) \Bar{f}(x,y) \Bar{g}(y) dx dy. 
    \end{equation*}
    Note that $\Bar{f}$ is only a continuous function if $f$ is zero on the boundary of $\Omega$. 
    However, the reason that~\cref{cor:HW} requires continuity is to apply Mercer's theorem to find an eigendecomposition of $f$.
    Since $f$ is continuous on $\Omega$ and $\Bar{f}$ is formed from blocks of $f$, we can still apply Mercer's theorem to $f$ and then form $\Bar{f}$ by zero-extending the eigenfunctions of $f$ appropriately. Thus, we can use the Hanson--Wright inequality even though $\Bar{f}$ is discontinuous.~\cref{cor:HW} tells us that
    \begin{equation} \label{eq:bound1}
        \mathbb{P}\left( |\phi_{\Bar{f}}(\Bar{g}) - \E[\phi_{\Bar{f}}(\Bar{g})]| \geq t  \right) 
        \leq 4 \exp \left(- \frac{t^2}{ 80\|\bar{f}\|_{L^2}^2 \|\bar{K}\|_{op}^2 +  16\|\bar{f}\|_{op} \|\bar{K}\|_{op} t }  \right),
    \end{equation}
    where $\phi_{\Bar{f}}(\Bar{g}) = \int_{\Bar{\Omega}}\int_{\Bar{\Omega}} \Bar{g}(x) \Bar{f}(x,y) \Bar{g}(y) dx dy$.
    Let $H_m(f)$ be defined as in~\cref{eq:hutch1}. 
    Furthermore, note that $\|\Bar{f}\|_{L^2}^2 = m \|f\|_{L^2}^2$, $\|\Bar{f }\|_{op} = \|f\|_{op}$, and $\|\Bar{K}\|_{op} = \|K\|_{op}$.
    Then, if we let $t' = t/m$, we can rewrite~\cref{eq:bound1} as 
    \begin{equation*}
        \mathbb{P}\left( \left|H_m(f) - \E[H_m(f)] \right| \geq t'  \right) \leq 4 \exp \left(- \frac{m t'^2}{80 \|f\|_{L^2}^2 \|K\|_{op}^2 +  16 \|f\|_{op} \|K\|_{op} t' }  \right).
    \end{equation*}
    Choose $t' = \|K\|_{op} \|f\|_{L^2} \sqrt{96 \log(4/\delta)/m}$. Then notice that since $m \geq 96 \log(4/\delta)$ and $\|f\|_{op} \leq \|f\|_{L^2}$, the second term in the denominator satisfies $\|f\|_{op} \|K\|_{op} t' \leq \|f\|_{L^2}^2 \|K\|_{op}^2$. Thus,
    \begin{equation*}
       \mathbb{P} \left( \left| H_m(f) - \E[H_m(f)] \right| \geq \|K\|_{op} \|f\|_{L^2} \sqrt{96 \log(4/\delta)/m}  \right) \leq 4 \exp \left(-\log (4/\delta)  \right) = \delta.
    \end{equation*}
\end{proof}

We can combine the result in~\cref{thm:ContHutch1} with our bound for $\E[H_m(f)]$ to obtain a probabilistic bound in terms of $\mathrm{tr}(f)$ when $K=K_{SE}$. Combining~\cref{thm:expectation1},~\cref{thm:ContHutch1}, and the triangle inequality establishes that $H_m(f)$ achieves an error of $\epsilon\|f\|_{L^2}$ with probability $1-\delta$ provided that the number of samples satisfies $m=\mathcal{O}(\log(4/\delta)/\epsilon^2)$ and the length scale satisfes $\ell=\mathcal{O}(\min\{d(\epsilon),\epsilon\})$.

\begin{corollary} \label{cor:ContHutch2}
     Let $f: \Omega \times \Omega \to \mathbb{R}$ be a continuous function with $\Omega = [a, b]$ and $-\infty < a < b < \infty$, where $f$ can be written as $f = f_1 - f_2$ for some continuous PSD functions $f_1, f_2: \Omega \times \Omega \to \R$.
    Let 
    \begin{align*}
         m \geq \max \left\{ \frac{384 \|K\|_{op}^2 \log(4/\delta)}{\varepsilon^2}, 96 \log(4/\delta) \right\}, \quad 
         \ell < \min \left \{ \frac{d(\varepsilon)}{\sqrt{2 \log\left( \frac{16 \|f\|_{\infty} (b-a)}{\varepsilon \|f\|_{L^2}} \right)}}, \frac{5 \varepsilon \|f\|_{L^2}}{52 \|f\|_\infty} \right\},
    \end{align*}
    where $d = d(\varepsilon) >0$ is such that $|f(x,y) - f(x,z)| \leq \frac{\varepsilon \|f\|_{L^2}}{8(b-a)}$ for all $y,z \in \Omega$ such that $|y-z| < d$.
    Then with probability $\geq 1-\delta$, 
    \begin{equation*} \label{eq:ContHutch}
        \left| H_m(f) - \mathrm{tr}(f) \right| \leq \varepsilon \|f\|_{L^2}.
    \end{equation*}
    where $H_m(f)$ is the continuous Hutchinson's estimate of $\mathrm{tr}(f)$ computed with $m$ probe functions drawn from $\mathcal{GP}(0, K_{SE})$
    where $K_{SE}$ has length-scale parameter $\ell$.
\end{corollary}
When $f$ is symmetric PSD, we have that $\|f\|_{L^2}^2 = \sum_i \lambda_i^2 \leq \lambda_1 \sum_i \lambda_i = \lambda_i \mathrm{tr}(f) \leq \mathrm{tr}(f)^2$, where $\lambda_i$ is the $i$th eigenvalue of $F$. Thus, for symmetric PSD functions, the result above gives a relative accuracy for the continuous Hutchinson's estimate of the trace of $f$. 

\section{ContHutch++} \label{sec:ContHutch++}
For continuous, symmetric PSD functions $f$, we can improve the bound for Hutchinson's estimator by projecting off a low-rank approximation of $F$ corresponding to its first few largest eigenvalues. To understand why, note that Hutchinson's estimator error bound for symmetric PSD functions is only tight when the first few eigenvalues of $F$ are much larger than the rest so that $\mathrm{tr}(f) \approx \|f\|_{L^2}$. Thus, if we project off the first few eigenvalues of $F$, we can use Hutchinson's estimator to estimate the trace of the resulting function with smaller eigenvalues. We use the continuous randomized range finder~\cite{ContRSVD} to compute the approximate projection from the first few eigenfunctions. 

We begin with a uniform approximation bound for a symmetric, PSD, continuous function by its best rank $k$ approximation (e.g., its truncated SVD) in terms of $\mathrm{tr}(f)$. This is a continuous analogue of a key oservation used to obtain the improved scaling of Hutch++ in the matrix case~\cite[Lemma~3]{Hutch++}.
\begin{theorem} \label{thm:rank_k}
    Let $F_k$ be the best rank $k$ approximation to the integral operator $F$ defined by a symmetric, PSD, continuous function $f$. Then,
\[
        \|F - F_k\|_{HS} \leq \frac{1}{\sqrt{k}} \mathrm{tr}(f).
\]
\end{theorem}
\begin{proof}
    Since $f$ is symmetric PSD and continuous, its eigendecomposition exists and converges absolutely and uniformly by Mercer's Theorem. Let $\lambda_1 \geq \lambda_2 \geq ... \geq 0$ be the eigenvalues of $F$.
    Thus, we have that 
    \begin{equation*}
        \lambda_{k+1} \leq \frac{1}{k} \sum_{i=1}^k \lambda_i \leq \frac{1}{k} \mathrm{tr}(f).
    \end{equation*}
    We can use this result to bound the error in the best rank $k$ approximation by 
    \begin{align*}
        \|F - F_k \|_{HS}^2 &= \sum_{i=k+1}^\infty \lambda_i^2 
        \leq \lambda_{k+1} \sum_{i=k+1}^\infty \lambda_i 
        \leq \frac{1}{k} \mathrm{tr}(f) \sum_{i=k+1}^\infty \lambda_i 
        \leq \frac{1}{k} \mathrm{tr}(f)^2.
    \end{align*}
\end{proof} 

Given a decomposition of $\Tr(f)$ such that $\Tr(f) = \Tr(f_1) + \Tr(f_2)$ where the $L^2$ norm of $f_2$ is bounded with high probability, we show that the approximation of $\Tr(f)$ given by computing $\Tr(f_1)$ exactly and $\Tr(f_2)$ using the continuous Hutchinson's estimator gives an accurate trace estimate with high probability for sufficient choices of $m$ and $\ell$. We use $\lceil a\rceil$ to denote the smallest integer $\geq a$.

\begin{theorem} \label{thm:Z1}
    Let $f:\Omega \times \Omega \to \mathbb{R}$ be continuous, symmetric, and PSD, $\varepsilon > 0$, $0 < \delta < 1$, and $m, k \in \mathbb{N}$. Let $f_1, f_2 : \Omega \times \Omega \to \mathbb{R}$ be functions that satisfy $\mathrm{tr}(f) = \mathrm{tr}(f_1) + \mathrm{tr}(f_2)$ and $\|f_2\|_{L^2} \leq \eta \|F - F_k \|_{HS}$ with probability $1 - \delta/2$ for some $\eta$, where $F_k$ denotes the best rank $k$ approximation to $F$.
    Let $k = \lceil\sqrt{\log(8/\delta)}/\varepsilon\rceil$ and $d = d(\varepsilon) >0$ satisfy $|f_2(x,y) - f_2(x,z)| \leq \frac{\varepsilon \|f_2\|_{L^2}}{8(b-a)}$ for all $y,z \in \Omega$ such that $|y-z| < d$.
    If 
    \begin{align*}
        m \geq \max \left\{\frac{384 \eta^2}{\|f\|_{L^2}^2}\left\lceil\frac{\sqrt{\log(8/\delta)}}{\varepsilon}\right\rceil, 96 \log(8/\delta) \right\},
        \quad 
        \ell < \min \left\{ \frac{d \left( \varepsilon\sqrt{k}/\eta \right)}{\sqrt{2 \log \left(\frac{16 \eta \|f_2\|_{\infty} (b-a)}{ \varepsilon\sqrt{k} \|f_2\|_{L^2}} \right)}}, 
        \frac{5 \varepsilon \sqrt{k} \|f_2\|_{L^2}}{52 \eta \|f_2\|_\infty} \right\},
    \end{align*}
     then $Z = \mathrm{tr}(f_1) + H_m(f_2)$, where $H_m$ is the Hutchinson's estimator with length-scale parameter $\ell$, satisfies
\[
        \mathbb{P} \big( |Z - \mathrm{tr}(f)| \geq \varepsilon \Tr(f) \big) \leq \delta.
\]
\end{theorem}
\begin{proof}
    We have
    \begin{equation*}
        |Z - \mathrm{tr}(f)| = |H_m(f_2) - \mathrm{tr}(f_2)|
        \leq \frac{\varepsilon \sqrt{k}}{\eta} \|f_2\|_{L^2},
    \end{equation*}
    with probability $\geq 1 - \delta/2$, where the inequality comes from the continuous Hutchinson's bound in~\cref{eq:ContHutch} with target accuracy $\varepsilon\sqrt{k}/\eta$. Then, by our assumption on $f_2$, we have
    \begin{align*}
        |Z - \mathrm{tr}(f)| &\leq \varepsilon \sqrt{k} \|F - F_k\|_{HS},
    \end{align*}
    with probability $\geq 1 - \delta$.
    Finally, by~\cref{thm:rank_k}, we can rewrite the bound in terms of $\Tr(f)$:
    \begin{equation*}
        \mathbb{P} \left( |Z - \mathrm{tr}(f)| \leq \varepsilon \Tr(f) \right) \geq 1 - \delta.
    \end{equation*}
\end{proof}

We now choose a decomposition $F = F_1 + F_2$ using the continuous randomized range finder and derive an error bound for the $\ContHutch++$ trace estimate of $f$, given by $Z$. Let $Q$ be the $\Omega \times n$ quasimatrix from the continuous randomized rangefinder applied to $F$.  The range of $Q$ approximates the range of $F$ with high probability. In particular, the recent result in~\cite[Theorem 3.3]{persson2024randomized} implies that the continuous randomized range finder applied to $F$ with target rank $k \geq 4$, oversampling parameter $p = k$, and failure probability controls $t,s \geq 1$, satisfies the error bound\footnote{The paper~\cite{persson2024randomized} deals primarily with the Nystrom approximation, but bounds for the randomized SVD of self-adjoint positive operators follow immediately from bounds for the Nystrom approximation (see Remark~2.5 of~\cite{persson2024randomized}).}
\begin{equation} \label{eq:cont_rsvd}
    \|F - QQ^T F\|_{HS} \leq \left[1 + 2\delta_k^{(\rm Tr)} + 3t^2\beta_k^{(\rm Tr)} + \frac{4 k }{(k+1)^2} e^2t^2s^2\beta_k^{(\rm Op)} \right] \left( \sum_{i = k+1}^\infty \sigma_i^2 \right)^{1/2},
\end{equation}
with probability $\geq 1 - 3t^{-p} - e^{-s^2/2}$. Here, $\sigma_1,\sigma_2,\sigma_3,\ldots$ denote the eigenvalues of $F$ and $\delta_k^{(\rm Tr)}$, $\beta_k^{(\rm Tr)}$, and $\beta_k^{(\rm Op)}$ are constants defined as follows. Let $F^TF = (V_1\quad V_2){\rm diag}(\Sigma_1,\Sigma_2)(V_1^T\quad V_2^T)^T$ represent the eigenvalue decomposition of $F^TF$ in quasimatrix form, partitioned so that $V_1$ is the quasimatrix of the first $k$ eigenvectors. Then, let $\hat K_{ij} = V_i^T K V_j$ and denote the Schur complement of the $(1,1)$ block in the $2\times 2$ block operator by $\hat K_{22,1} = \hat K_{22}-\hat K_{21}\hat K_{11}^{-1}\hat K_{21}^T$. For $\zeta \in \{\rm Tr, \rm Op\}$, define the constants
\begin{equation}
\delta_k^{(\zeta)} = \frac{\|\Sigma_2^{1/2}\hat K_{22,1}\Sigma_2^{1/2}\|_\zeta}{\|\Sigma_2\|_\zeta}\|\hat K_{11}^{-1}\|_2,
\qquad 
\beta_k^{(\zeta)} = \frac{\|\Sigma_2^{1/2}\hat K_{21}\hat K_{11}^{-1}\hat K_{21}^T\Sigma_2^{1/2}\|_\zeta}{\|\Sigma_2\|_\zeta}\|\hat K_{11}^{-1}\|_2. 
\end{equation}
\Cref{eq:cont_rsvd} controls the accuracy of the continuous randomized range finder relative to the best rank $k$ approximation of $F$. The behavior of the constants $\delta_k^{(\rm Tr)}$, $\beta_k^{(\rm Tr)}$, and $\beta_k^{(\rm Op)}$ is particularly interesting when $K=K_{SE}$ with small smoothing parameter $\ell>0$ (see the discussion after~\cref{cor:ContHutch++}).

Using the continuous randomized range finder, we can write $F$ as 
\begin{equation*}
    F = (QQ^T (F QQ^T)) + \left(FQQ^T + QQ^T F - 2QQ^T (F QQ^T) + (I - QQ^T) (F (I - QQ^T)) \right).
\end{equation*}
Since $F$ is the integral operator with continuous, PSD, symmetric kernel $f$, $\Tr(F) = \Tr(f)$.
If we take the trace of the first and the second terms above, we find that they are equal to $\mathrm{tr}(QQ^T F QQ^T)$ and $\mathrm{tr}((I - QQ^T) F (I - QQ^T))$, respectively. Thus, we choose 
\[
\begin{aligned}
    F_1 = QQ^T F QQ^T, \qquad\text{and}\qquad
    F_2 = (I - QQ^T) (F (I - QQ^T)).
\end{aligned}
\]
This choice ensures that the functions defining $F_1$ and $F_2$ are symmetric PSD. Furthermore, $F_1$ and $F_2$ are both trace class since their traces are bounded by $\Tr(f)$.
Since traces are cyclic, we find that
\begin{align*}
    \mathrm{tr}(F_1) = \Tr(Q^T FQ),\qquad\text{and}\qquad
    \Tr(F_2) = \Tr((I - QQ^T)(F (I - QQ^T))).
 \end{align*}
 We directly compute the quantity $\Tr(F_1)$ since $Q^T F Q$ is an $n \times n$ matrix. We use Hutchinson's estimator to approximate $\Tr(F_2)$. The trace of $F_1$ requires $n$ operator-function products, i.e.,
 \begin{equation*}
     \Tr(F_1) = \sum_{i=1}^n \int_{\Omega} \int_\Omega q_i(x) f(x,y) q_i(y) dxdy,
 \end{equation*}

We can use the continuous randomized range finder bound, given in~\cref{eq:cont_rsvd}, to bound the error on applying the continuous Hutchinson's estimator to $F_2$. First, note that $\Tr(F_2) = \Tr((I - QQ^T)F)$ since the trace is cyclic and the projection $I - QQ^T$ has the property that $(I - QQ^T) = (I - QQ^T)^2$. Thus, computing the continuous Hutchinson's estimate of $\Tr(F_2)$ gives a bound in terms of $\|F - QQ^T F\|_{HS}$, which we can bound using the continuous randomized range finder. That is, if we draw $k+p=2k$ functions from the Gaussian process $\mathcal{GP}(0,K_{\rm SE})$, we can apply~\cref{thm:Z1} to $F_1$ and $F_2$, with 
 \begin{equation} \label{eq:eta}
     \eta = 1 + 2\delta_k^{(\rm Tr)} + 3t^2\beta_k^{(\rm Tr)} + \frac{4k}{(k+1)^2} e^2t^2s^2\beta_k^{(\rm Op)} .
 \end{equation}
 In particular, if we choose $k = p = \mathcal{O}(1/\varepsilon)$, $t=2$, and $s\geq\sqrt{2\log(2/\delta)}$ in~\cref{eq:cont_rsvd}, then we can apply~\cref{thm:Z1} to obtain an analogue of~\cref{thm:Hutch++} and bound the number of operator-function products required by $\ContHutch++$. The following corollary demonstrates that $\ContHutch++$ typically achieves a relative error $\leq \epsilon$ with probability $1-\delta$ provided that the total number of probe samples scales like $m=\mathcal{O}(\log(8/\delta)/\varepsilon)$ and the Hutchinson length scale satisfies $\ell=\mathcal{O}(\min\{d(\sqrt{\varepsilon}),\sqrt{\varepsilon}\})$.
 
\begin{figure}
    \centering
    \begin{minipage}{0.48\textwidth}
        \begin{overpic}[width=\textwidth]{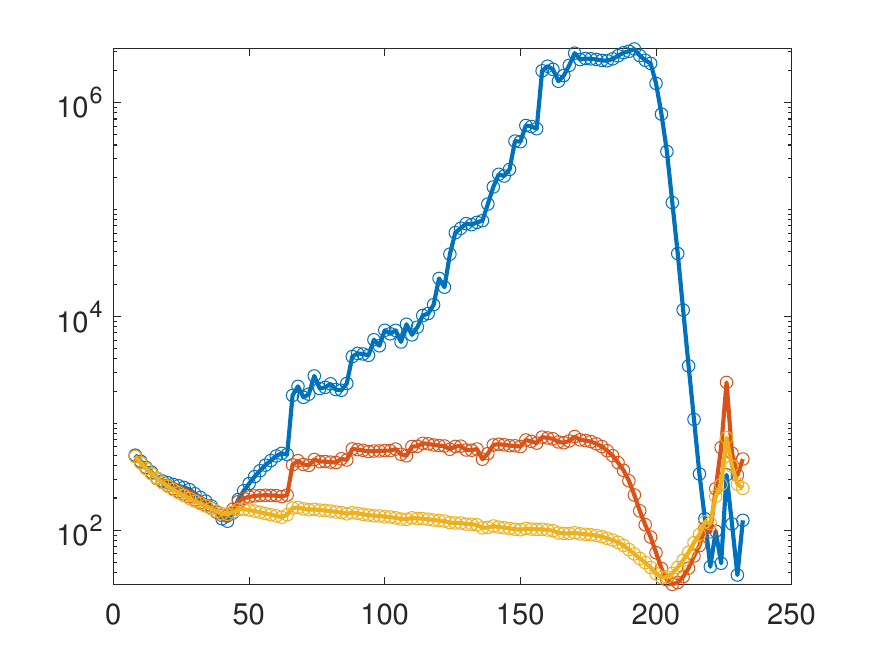}
        \put (45,-4) {$\displaystyle k$}
        \put (0,35) {\rotatebox{90} {$\displaystyle \eta$}}
        \end{overpic}
    \end{minipage}
    \begin{minipage}{0.48\textwidth}
        \begin{overpic}[width=\textwidth]{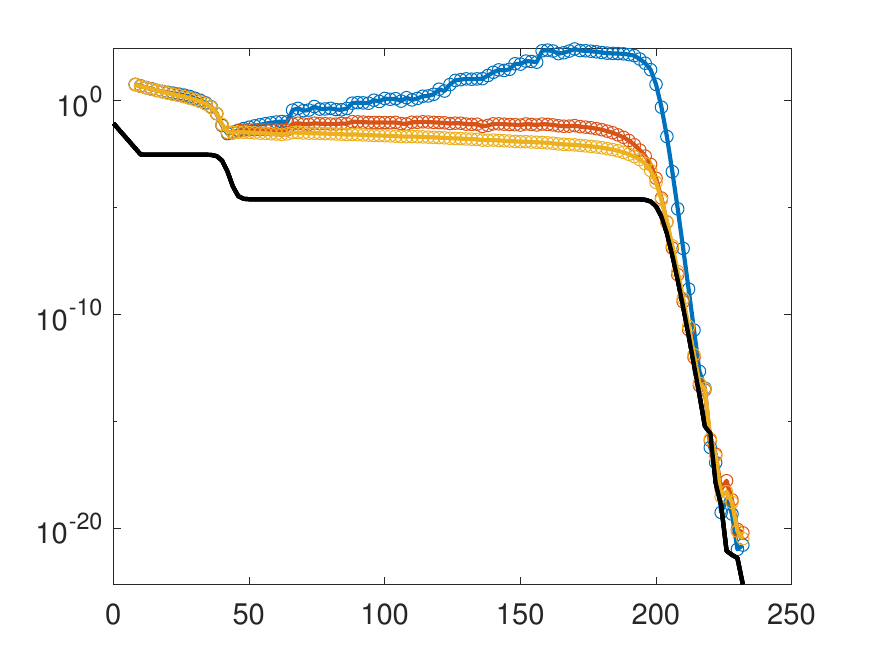}
        \put (48,-4) {$\displaystyle k$}
        \put (-1,24) {\rotatebox{90} {$\displaystyle \eta\|F-QQ^TF\|$}}
        \end{overpic}
    \end{minipage}
    \caption{\label{fig:rsvd_scaling} In the left panel, the constant $\eta$ in~\cref{cor:ContHutch++} is calculated for the trace-class kernel in~\cref{eqn:kernel2} and plotted against $k$ (the total number of probe vectors is $k+p=2k$) used in the continuous randomized range finder for $\ell=0.04$ (blue), $\ell=0.02$ (red), and $\ell=0.01$ (yellow) in the squared exponential covariance kernel $K_{\rm SE}$. In the right panel, the corresponding relative error bound for the rangefinder is plotted against $k$ for $\ell=0.04$ (blue), $\ell=0.02$ (red), and $\ell=0.01$ (yellow) and compared with the best low-rank approximation error (black).}
\end{figure}

For rigor in the following result, we make a slight modification to~\cref{alg:ContHutch++}. Suppose that the continuous randomized range finder is applied with $k+p$ probe vectors drawn from $\mathcal{GP}(0,K_{SE}^{(1)})$ with parameter $\ell_1>0$ while the Hutchinson's part of the estimate is carried out with $n$ probe vectors drawn from a second gaussian process, $\mathcal{GP}(0,K_{SE}^{(2)})$ with parameter $\ell_2>0$. The total number of operator-function products required is then $m=2(k+p)+n$. This setup allows us to circumvent the dependence of the parameter $\eta$ in~\cref{eq:eta} on properties of the kernel $K_{SE}^{(1)}$ and rigorously choose parameters required to obtain accuracy $\leq \varepsilon$ with high probability. In practice, we believe that using two separate gaussian processes in $\ContHutch++$ is unnecessary (see the discussion after~\cref{cor:ContHutch++}).
 
 \begin{corollary} \label{cor:ContHutch++}
Let $f:\Omega \times \Omega \to \mathbb{R}$ be continuous, symmetric, and PSD, $\varepsilon>0$, $0<\delta<1$, and $m, k \in \mathbb{N}$. Let $d = d(\varepsilon)$ be such that $|f(x,y) - f(x,z)| \leq \frac{\varepsilon \|f\|_{L^2}}{8(b-a)}$ for all $y,z \in \Omega$ such that $|y-z| < d$. Let $Q$ be the output of the continuous randomized rangefinder applied to $F$ with $\mathcal{GP}(0,K_{SE}^{(l_1)})$ and probe parameters $k=\lceil\sqrt{\log(8/\delta)}/\varepsilon\rceil$, $p=\max\{k,4\}$ and let $\eta$ be defined in~\cref{eq:eta} with $t=2$ and $s\geq\sqrt{2\log(2/\delta)}$. If the continuous Hutchinson's estimate of the tail $(I-QQ^T)F(I-QQ^T)$ is calculated with $n$ probe vectors and smoothing parameter $\ell_2$ that satisfy
    \begin{align*}
        n \geq \max \left\{\frac{384 \eta^2}{\|f\|_{L^2}^2}\left\lceil\frac{\sqrt{\log(8/\delta)}}{\varepsilon}\right\rceil, 96 \log(8/\delta) \right\},
        \quad
        \ell_2 < \min \left\{ \frac{d \left( \varepsilon\sqrt{k}/\eta \right)}{\sqrt{2 \log \left(\frac{16 \eta \|f\|_{\infty} (b-a)}{ \varepsilon\sqrt{k} \|f\|_{L^2}} \right)}}, 
        \frac{5 \varepsilon \sqrt{k} \|f\|_{L^2}}{52 \eta \|f\|_\infty} \right\},
    \end{align*}
     then $|\left[\mathrm{tr}(Q^TFQ)+H_m((I-QQ^T)F(I-QQ^T))\right] - \mathrm{tr}(f)| \leq \varepsilon \Tr(f)$ with probability $\geq 1 - \delta$.
\end{corollary}

When the constant $\eta$ from~\cref{eq:eta} in~\cref{cor:ContHutch++} remains bounded independently of $k$, \cref{cor:ContHutch++} shows that the number of operator-function products required to achieve target accuracy $\epsilon$ with probability $1-\delta$ scales like $m=\mathcal{O}(\log(8/\delta)/\epsilon)$. However, $\eta$ does depend on $k=p$ and $\ell_1$ through the constants $\delta_k^{(\rm Tr)}$, $\beta_k^{(\rm Tr)}$, and $\beta_k^{(\rm Op)}$. In practice, for sufficiently small $\ell_1$, we observe that $\eta$ remains bounded over practical ranges of probe vectors $k$ and the departure of the continuous randomized range finder from the best rank $k$ approximation is also bounded, leading to the observed speed-up of $\ContHutch++$. For the trace-class kernel in~\cref{eqn:kernel2},~\cref{fig:rsvd_scaling} illustrates a typical dependence of $\eta$ (left panel) and the error in the continuous randomized range finder (right panel) as a function of $k=p$ for decreasing values of $\ell$ ($\ell=0.04$ (blue), $\ell=0.02$ (red), and $\ell=0.01$ (yellow)). When there are plateaus in the spectrum of $F$, $\eta$ may grow rapidly when $\ell$ is not sufficiently small but, as $\ell$ decreases, the growth of $\eta$ over these plateaus is brought under control. When the singular values of $F$ decay rapidly, the continuous randomized range finder tracks closely with the best rank $k$ approximation. For this reason, one may use $m=\mathcal{O}(\log(8/\delta)/\epsilon)$ probe vectors drawn from a single Gaussian process in~\cref{alg:ContHutch++} with $\ell>0$ and observe the expected error scalings when $\ell$ is sufficiently small.

\section{Numerical Experiments}\label{sec:num_exp}
We now apply $\ContHutch++$ to two infinite-dimensional trace estimation problems that arise in physical applications. First, we combine $\ContHutch++$ with rational smoothing kernels and demonstrate a new fast and accurate algorithm to compute the density-of-states (DOS) of Schr\"odinger operators. Second, we illustrate how $\ContHutch++$ can efficiently and robustly calculate mean-field quantities that appear in photonic design problems with incoherent sources.

\subsection{Example 1: Density-of-states for Schr\"odinger operators}\label{sec:DOS}

\begin{figure}
    \centering
    \begin{minipage}{0.48\textwidth}
        \begin{overpic}[width=\textwidth]{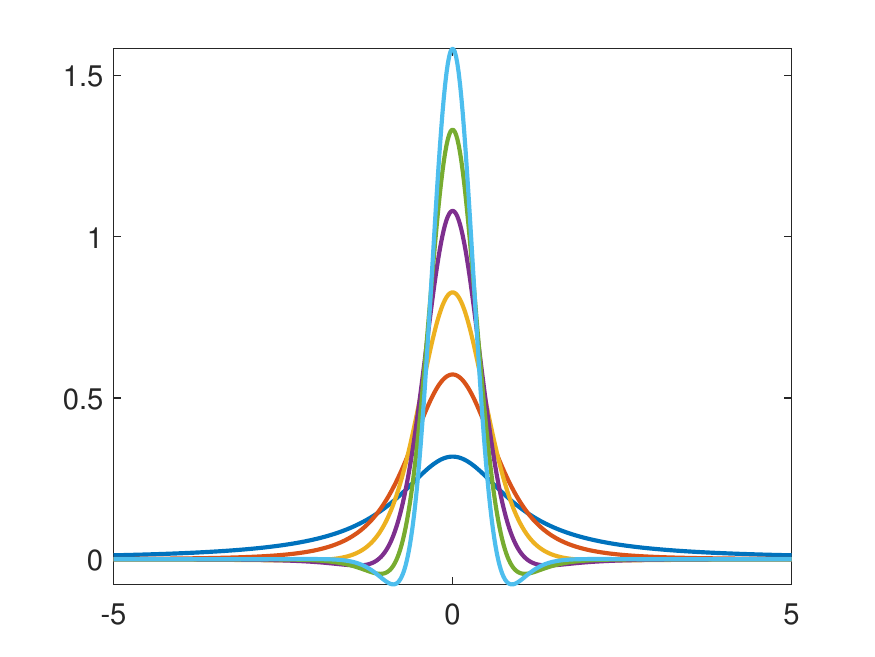}
        \put (50,-2) {$\displaystyle x$}
        \put (-1,38) {\rotatebox{90} {$\displaystyle f(x)$}}
        \put (30,72) {$\displaystyle \text{Convolution Kernels}$}
        \put (62,20) {{$\displaystyle \text{Poisson}$}}
        \put (62,28) {{$\displaystyle K = 2$}}
        \put (62,37) {{$\displaystyle K = 3$}}
        \put (62,46) {{$\displaystyle K = 4$}}
        \put (62,55) {{$\displaystyle K = 5$}}
        \put (62,64) {{$\displaystyle K = 6$}}
        \end{overpic}
    \end{minipage}
    \begin{minipage}{0.48\textwidth}
        \begin{overpic}[width=\textwidth]{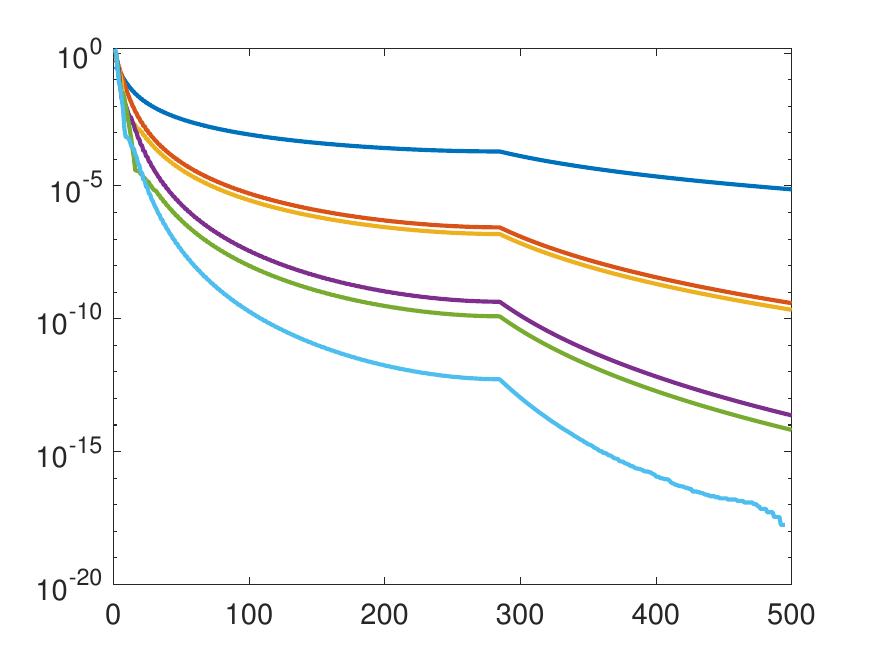}
        \put (50,-2) {$\displaystyle k$}
        \put (-1,38) {\rotatebox{90} {$\displaystyle \sigma_k$}}
        \put (24,72) {$\displaystyle \text{Singular Values of Operators}$}
        \put (73,16) {\rotatebox{-17} {$\displaystyle K = 6$}}
        \put (73,26) {\rotatebox{-16} {$\displaystyle K = 5$}}
        \put (74,33) {\rotatebox{-15} {$\displaystyle K = 4$}}
        \put (75,38) {\rotatebox{-12} {$\displaystyle K = 3$}}
        \put (76,44) {\rotatebox{-10} {$\displaystyle K = 2$}}
        \put (68,57) {\rotatebox{-7} {$\displaystyle \text{Poisson}$}}
        \end{overpic}
    \end{minipage}
    \caption{\label{fig:dos1} The first six rational convolution kernels based on equispaced points are plotted in the left panel. The eigenvalues of the corresponding operators in~\cref{eqn:dos2} with $\mathcal{L} = -\Delta$, $\lambda=10$, and $\sigma = 0.25$, are plotted in the right panel.}
\end{figure}
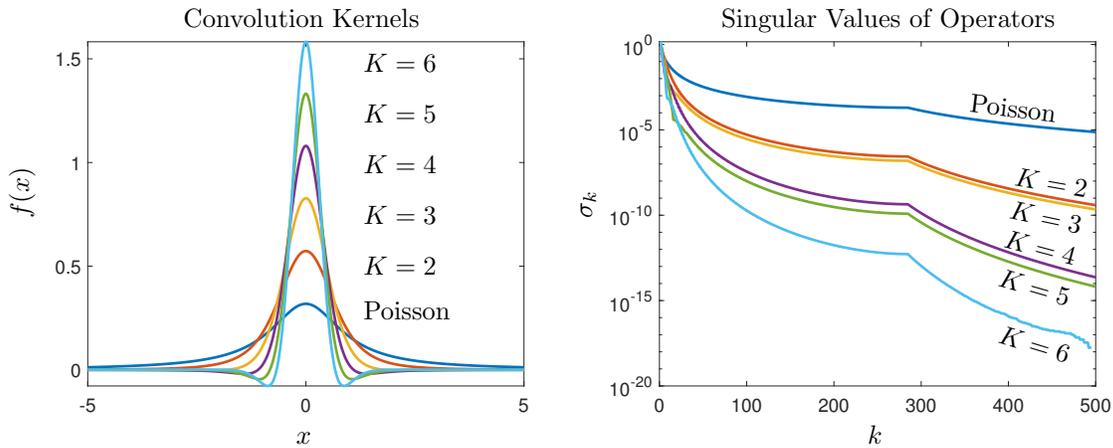

A non-dimensionalized Schr\"odinger operator on a finite interval of length $2L$ takes the form
\begin{equation}\label{eqn:Schrodinger}
    [\mathcal{L}u](x) = [-\Delta + v(x)]u(x), \qquad x \in[-L,L],
\end{equation}
where $\Delta$ is the Laplacian operator, $v(x)$ is a potential function, and $u(x)$ satisfies homogeneous Dirichlet or periodic boundary conditions on the boundary of the box. The eigenvalues of $\mathcal{L}$ govern the observable energy spectrum of a variety of quantum systems, and many important physical quantities can be calculated from their asymptotic density, which is known as the density-of-states and is defined as
\begin{equation}\label{eqn:dos}
    \rho_L(\lambda) = \frac{1}{2L}\sum_{k=1}^\infty \delta(\lambda - \lambda_k).
\end{equation}
Here, $\lambda_1,\lambda_2,\lambda_3,\ldots$ are the eigenvalues of $\mathcal{L}$ and $\delta(\cdot)$ is the Dirac delta measure with unit mass at the origin. The discrete measure $\rho_L(\lambda)$ often converges to an absolutely continuous measure with a piece-wise smooth density in the so-called ``thermodynamic limit" $L\rightarrow\infty$.

\begin{figure}
    \centering
    \begin{minipage}{0.48\textwidth}
        \begin{overpic}[width=\textwidth]{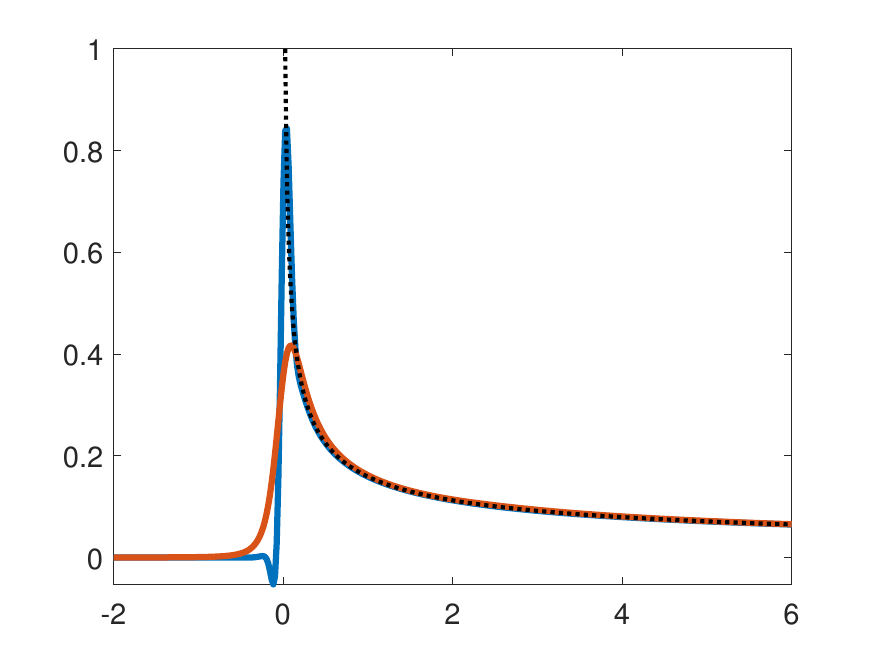}
        \put (50,-2) {$\displaystyle \lambda$}
        \put (-1,38) {\rotatebox{90} {$\displaystyle \rho_L^{(K)}$}}
        \end{overpic}
    \end{minipage}
    \begin{minipage}{0.48\textwidth}
        \begin{overpic}[width=\textwidth]{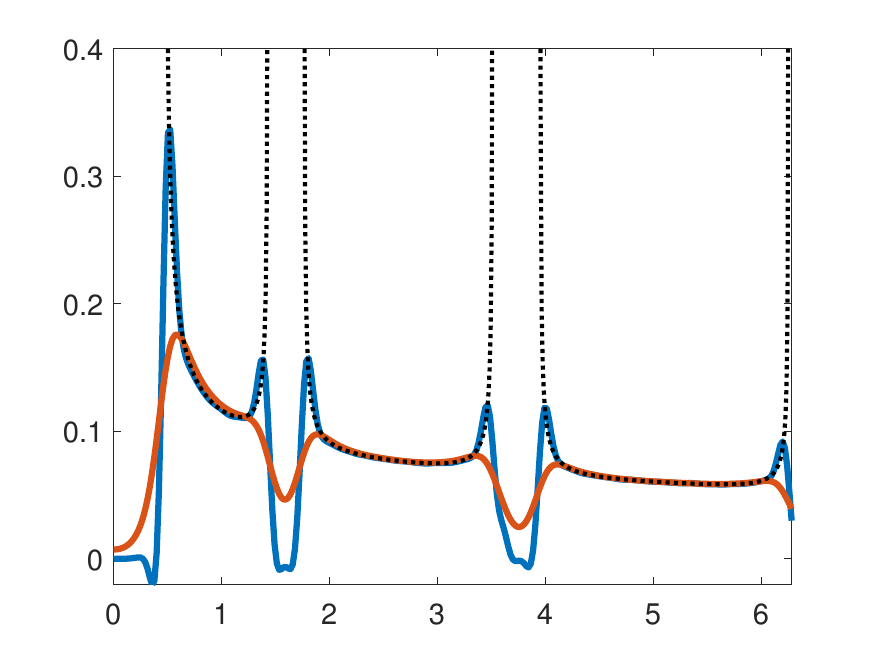}
        \put (50,-2) {$\displaystyle \lambda$}
        \put (-1,38) {\rotatebox{90} {$\displaystyle \rho_L^{(k)}$}}
        \end{overpic}
    \end{minipage}
    \caption{\label{fig:dos2} Smoothed approximate density-of-states for the free particle (left) and the Kronig--Penney model (right) were computed using rational kernels of order $K=2$ (red) and $K=8$ (blue). The smoothed approximations are compared with analytic formulas (dotted lines) for the density-of-states in the asymptotic limit $L\rightarrow\infty$.}
\end{figure}

The density-of-states can be approximated numerically by exploiting a link with the trace of the resolvent operator. After convolution with the Poisson kernel, a smoothed density-of-states is given by
\begin{equation}\label{eqn:dos1}
    \rho_L^\sigma(\lambda) = \frac{1}{2L}\sum_{k=1}^\infty\frac{\sigma}{(\lambda-\lambda_k)^2+\sigma^2} = \frac{1}{2L}{\rm tr}\left[{\rm Im}\,(\mathcal{L}-\lambda-i\sigma)^{-1}\right].
\end{equation}
As the smoothing parameter $\sigma\rightarrow 0$, $\rho_L^\sigma\rightarrow\rho_L$ weakly in the sense of measures. Numerically, $\sigma$ and the discretization of the resolvent in~\cref{eqn:dos1} must be refined with caution to avoid instability caused by finite-size effects, as discussed in the closely related context of spectral measures of self-adjoint operators~\cite{colbrook2021computing,dupuy2022finite}. Instead, we use $\ContHutch++$ to compute the trace of the resolvent, balancing the trace estimation target accuracy with asymptotically tight smoothing error bounds for $\sigma>0$~\cite{colbrook2021computing}.

For two reasons, achieving high-accuracy density-of-states calculations with the Poisson kernel is challenging. First, $\rho_L^\sigma$ converges slowly as $\sigma\rightarrow 0$, and smaller values of $\sigma$ require larger discretizations of the resolvent. Second, the filtered resolvent operator in~\cref{eqn:dos1} has a slowly decaying spectrum because of the slow decay of the Poisson kernel. Consequently, many operator-function products are required for an accurate trace estimate. We can mitigate both challenges by smoothing $\rho_L(\lambda)$ with higher-order rational convolution kernels, introduced for high-resolution computation of spectral measures, instead of the Poisson kernel~\cite{colbrook2021computing,colbrook2023computing}. A smoothed density-of-states obtained from a $K$th order rational kernel with simple poles $p_1,\ldots,p_K$ in the upper complex half-plane and residues $r_1,\ldots,r_K$ takes the form
\begin{equation}\label{eqn:dos2}
\rho_L^{(K)}(\lambda) = \frac{1}{2L}{\rm tr}\left[{\rm Im}\sum_{j=1}^K r_j(\mathcal{L}-\lambda+p_j\sigma)^{-1}\right].
\end{equation}
Six rational kernels with equispaced poles along ${\rm Im}(z)=1$ are plotted in~\cref{fig:dos1} (left) along with the eigenvalues associated with the trace estimation problems for the free particle (right), where $v(x) = 0$.

\begin{figure}
    \centering
    \begin{minipage}{0.48\textwidth}
        \begin{overpic}[width=\textwidth]{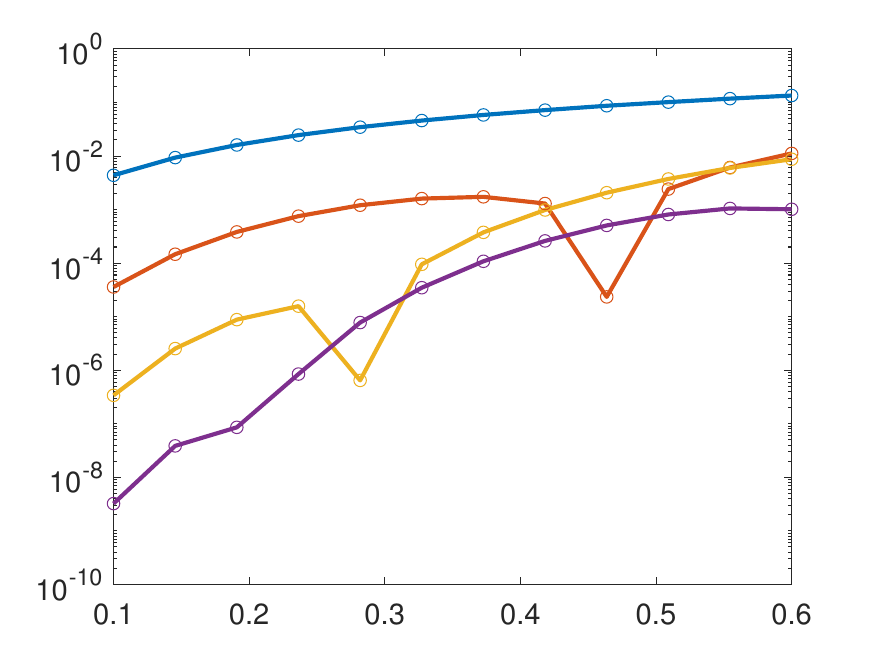}
        \put (50,-2) {$\sigma$}
        \put (0,15) {\rotatebox{90} {$\displaystyle |\mathbb{E}[H_m^{++}(f)]-tr(f)|/tr(f)$}}
        \put (15,58) {\rotatebox{10} {$\displaystyle K = 2$}}
        \put (15,47) {\rotatebox{18} {$\displaystyle K = 4$}}
        \put (15,36) {\rotatebox{24} {$\displaystyle K = 6$}}
        \put (15,24) {\rotatebox{28} {$\displaystyle K = 8$}}
        \end{overpic}
    \end{minipage}
    \begin{minipage}{0.48\textwidth}
        \begin{overpic}[width=\textwidth]{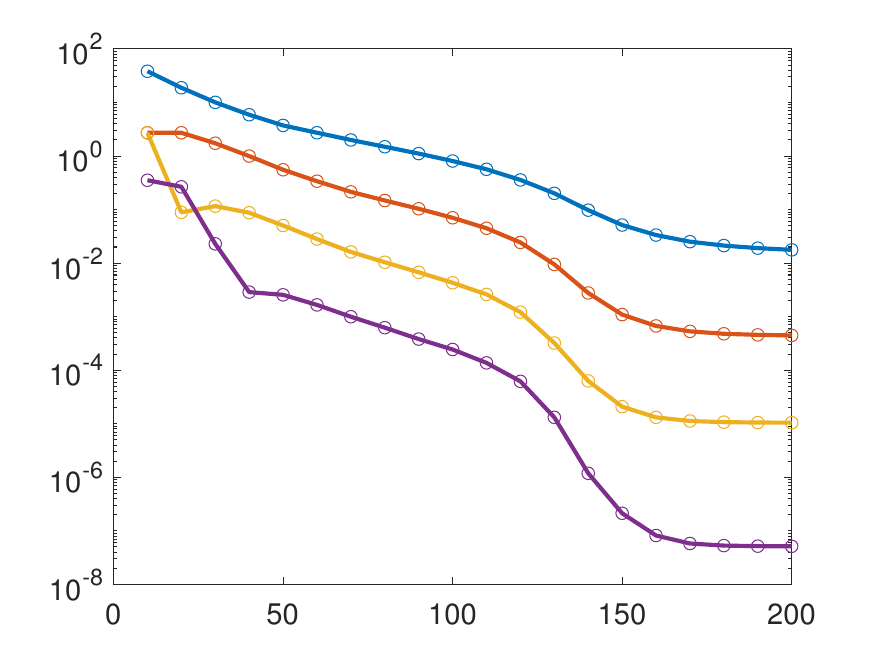}
        \put (48,-2) {$m/3$}
        \put (0,15) {\rotatebox{90} {$\displaystyle |\mathbb{E}[H_m^{++}(f)]-tr(f)|/tr(f)$}}
        \put (76,50) {\rotatebox{-4} {$\displaystyle K = 2$}}
        \put (76,39) {\rotatebox{-2} {$\displaystyle K = 4$}}
        \put (76,29) {\rotatebox{-2} {$\displaystyle K = 6$}}
        \put (76,15) {\rotatebox{-3} {$\displaystyle K = 8$}}
        \end{overpic}
    \end{minipage}
    \caption{\label{fig:dos3} Point-wise smoothing rates (left) and sample rates (right) of convergence of $\rho^{(K)}_{150}(1)$ for a free particle, with smoothing kernel orders $K=2,4,6,8$. Note that the relative errors in the right panel are plotted against $m/3$, the subspace dimension used for the randomized SVD and Hutchinson's estimate. The number of function samples was fixed at $m=600$ for the smoothing rate experiment (left), and the smoothing parameter was fixed at $\sigma=0.2$ for the sample rate experiment (right).  }
\end{figure}
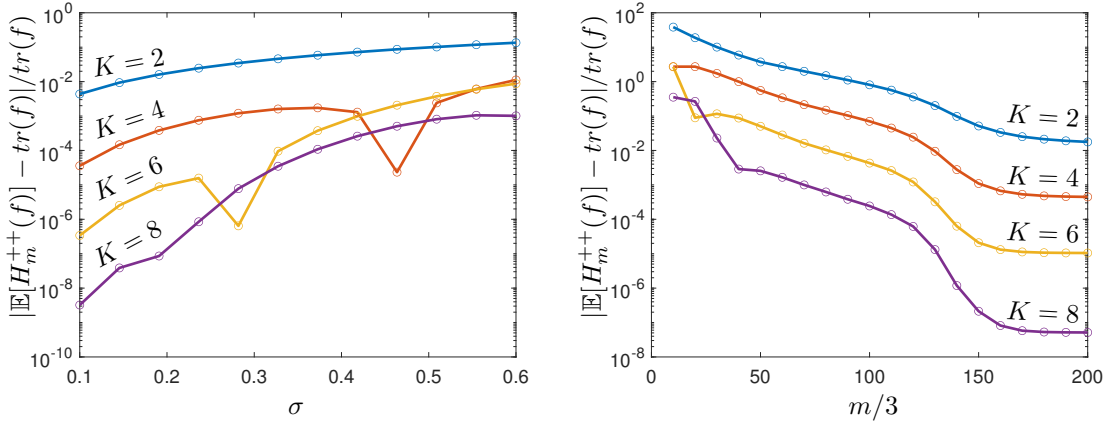

\Cref{fig:dos2} compares smoothed approximations to the density-of-states for a free particle (left) and the Kronig--Penney model (right) with analytic formulas for the exact thermodynamic limit (dotted line). The $K=8$ smoothing kernel (blue) gives much sharper results than the $K=2$ kernel (red) but takes negative values near endpoint singularities because the kernel is non-positive. A smoothing parameter of $\sigma=0.2$ is employed for both experiments, and the resolvent was approximated on intervals with $L=150$ and $L=20\pi$ for the free particle and the Kronig--Penney model, respectively.

In~\cref{fig:dos3}, the mean relative error in the density-of-states, evaluated at $\lambda=1$, for the free particle is plotted against the smoothing parameter $\sigma$ (left), and the number of operator-function products, $m$, used in $\ContHutch++$ (right). Note that we plot the relative error against $m/3$, which is the dimension of the subspace used for the randomized SVD and Hutchinson's estimate that are combined in the Hutch++ estimate. The domain size is $L=150$, and the random function samples in $\ContHutch++$ are drawn from a Gaussian process generated by the squared-exponential covariance kernel with parameter $\ell=2L\times 10^{-3}$. The resolvent action is computed with a maximum discretization size of $N=4100$. The order of convergence in the smoothing parameter increases with the order of the smoothing kernel. The accuracy also improves with the number of function samples until saturating at a limiting accuracy due to smoothing. With a fixed number of samples, the higher-order kernels are more accurate due to the faster spectral decay of the filtered operator (see~\cref{fig:dos1}).

\subsection{Example 2: Mean field intensity from incoherent sources}\label{sec:field_intensity}
A recent body of work in photonic design~\cite{PhysRevLett.107.114302,PhysRevB.92.134202,yao2022trace} formulates essential physical quantities, such as average field intensity and power emission, as the trace of a self-adjoint, semi-definite operator related to the solution of Maxwell's equations. Efficient optimization schemes leverage low-rank structure in the operator to optimize the trace without explicitly forming discretizations of the operator, which is prohibitively expensive. Interestingly, the spectral profile of this operator can depend on the material parameters, geometry, and separation between the current sources and the target emission region, some or all of which may vary during optimization. This section illustrates how $\ContHutch++$ can robustly calculate the mean-square field intensity in a dielectric tube induced by spatially uncorrelated currents.

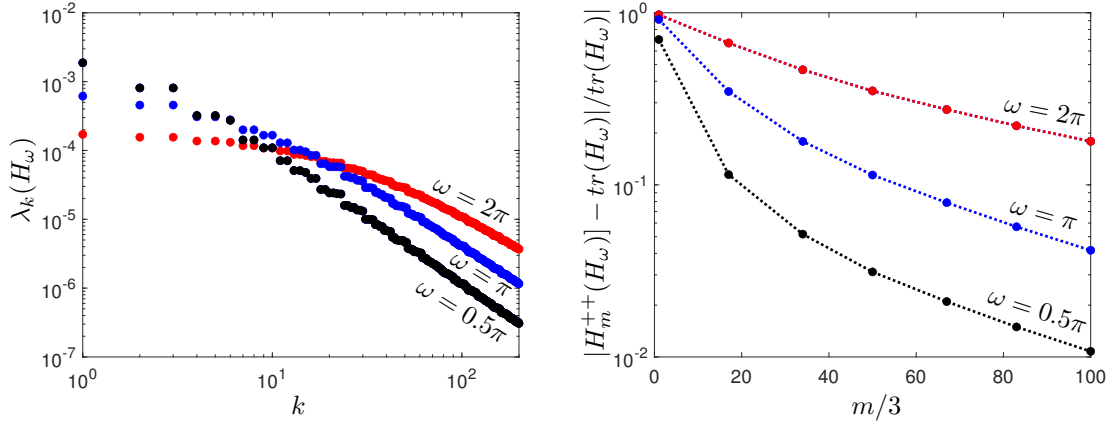
\begin{figure}
    \centering
    \begin{minipage}{0.48\textwidth}
        \begin{overpic}[width=\textwidth]{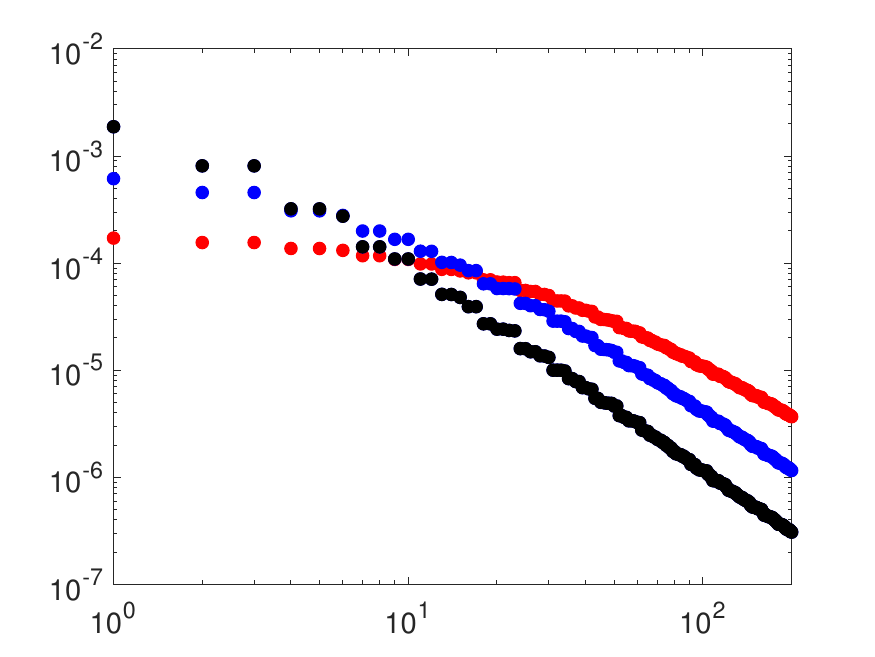}
        \put (50,-2) {$\displaystyle k$}
        \put (0,32) {\rotatebox{90} {$\displaystyle \lambda_k(H_\omega)$}}
        \put (75,38) {\rotatebox{-26} {$\displaystyle \omega = 2\pi$}}
        \put (78,25) {\rotatebox{-32} {$\displaystyle \omega = \pi$}}
        \put (72,20) {\rotatebox{-32} {$\displaystyle \omega = 0.5\pi$}}
        \end{overpic}
    \end{minipage}
    \begin{minipage}{0.48\textwidth}
        \begin{overpic}[width=\textwidth]{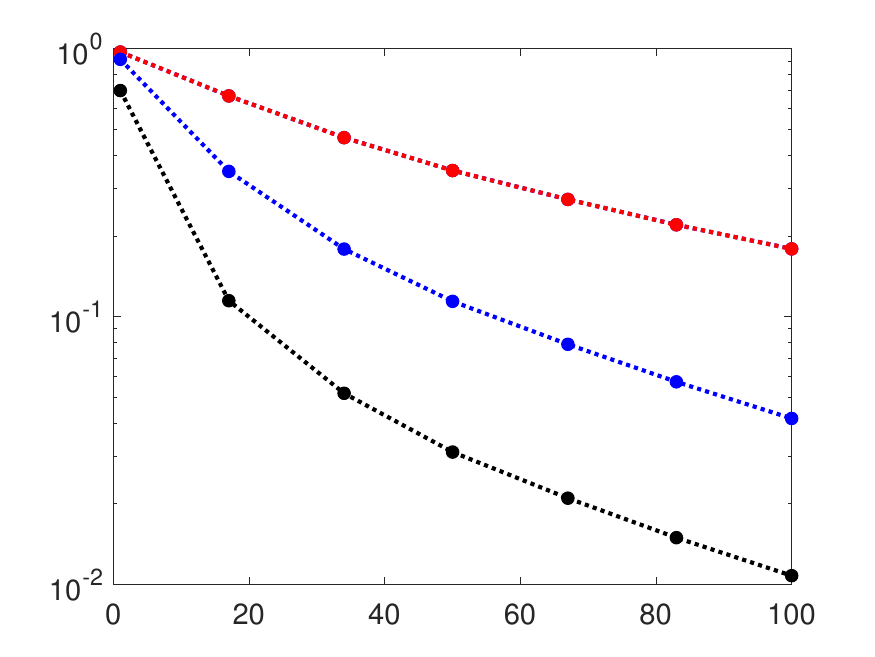}
        \put (48,-2) {$\displaystyle m/3$}
        \put (0,8) {\rotatebox{90} {$\displaystyle |H_m^{++}(H_\omega)]-tr(H_\omega)|/tr(H_\omega)|$}}
        \put (75,52) {\rotatebox{-11} {$\displaystyle \omega = 2\pi$}}
        \put (76,34) {\rotatebox{-16} {$\displaystyle \omega = \pi$}}
        \put (72,18) {\rotatebox{-18} {$\displaystyle \omega = 0.5\pi$}}
        \end{overpic}
    \end{minipage}
    \caption{\label{fig:spectrum_disk} Plotted in the left panel are the first $200$ eigenvalues of the operator in~\cref{eqn:mean_field_intensity}, whose trace corresponds to the mean field intensity due to spatially incoherent current sources with frequency $\omega$ in a dielectric tube with circular cross-section of radius $r=0.5.$ As the angular frequency $\omega$ increases from $\pi/2$ to $2\pi$, the leading spectrum of the operator becomes increasingly flat and takes longer to approach the asymptotic algebraic decay rate. In the right panel, we plot the relative error in the $\Hutch++$ trace estimate as the number of source samples increases from $m = 3,\ldots,300$.}
\end{figure}
 
Consider a $z$-invariant dielectric tube whose cross-section $\Omega\subset\mathbb{R}^2$ has a H\"older continuous boundary.
According to Maxwell's equations, a time-harmonic current with frequency $\omega$ and amplitude $b(x,y)$ in the $z$-direction produces an electric field $E_z$ in the $z$-direction satisfying (after nondimensionalization)
\begin{equation}\label{eqn:Maxwell2D}
    \Delta E_z - \omega^2\epsilon(x,y) E_z = i\omega b(x,y), \qquad \text{where}\qquad (x,y)\in\mathbb{R}^2.
\end{equation}
Here, $\epsilon(x,y)$ is the dielectric function, smoothly interpolating between the background's relative permittivity of $\epsilon_1=1$ and the dielectric's relative permittivity of $\epsilon_2=12$, so that
\begin{equation}\label{eqn:smoothed_perm}
\epsilon(x,y) = \epsilon_1 + \xi_\sigma(x,y)(\epsilon_2-\epsilon_1), \qquad\text{where}\qquad \xi_\sigma(x,y) = \frac{1}{(\sigma\sqrt{2\pi})^2}\int_\Omega e^{-\frac{(y-y')^2+(x-x')^2}{2\sigma^2}}\,dx'dy'.
\end{equation}
The interpolation parameter $\sigma>0$ controls the sharpness of the transition between the dielectric and background materials. In the limit of spatially uncorrelated currents, used to model fields in fluorescent materials and other processes of spontaneous emission, the mean-square field intensity is given by~\cite{yao2022trace}
\begin{equation}\label{eqn:mean_field_intensity}
\langle E_z \rangle = \mathbb{E}\left[\frac{1}{|\Omega|}\int_\Omega |E_z(x,y)|^2\,dxdy\right] = \frac{1}{|\Omega|}{\rm tr}\left[(A_\omega M_\Omega)^\dagger M_\Omega (A_\omega M_\Omega) \right].
\end{equation}
Here, $A_\omega$ represents the solution operator mapping the current source $b$ in the dielectric to the field $E_z$ in~\cref{eqn:Maxwell2D}. $M_\Omega$ represents multiplication by the smoothed characteristic function $\xi_\sigma$ in~\cref{eqn:smoothed_perm}. Note that the spectrum of the self-adjoint operator $H_\omega=(A_\omega M_\Omega)^\dagger M_\Omega (A_\omega M_\Omega)$ typically decays slowly because the target field and the current source are both supported in the dielectric cross-section, $\Omega$. \Cref{fig:spectrum_disk} (left panel) displays the first $200$ eigenvalues of $H_\omega$ for a circular cross-section $\Omega = \{(x,y)\in\mathbb{R}^2:x^2+y^2\leq (0.5)^2\}$ with angular frequency values $\omega = 0.5\pi$, $\pi$, and $2\pi$.

\begin{figure}[htb]
    \centering
\begin{minipage}{0.32\textwidth}
  \begin{overpic}[width=\textwidth]{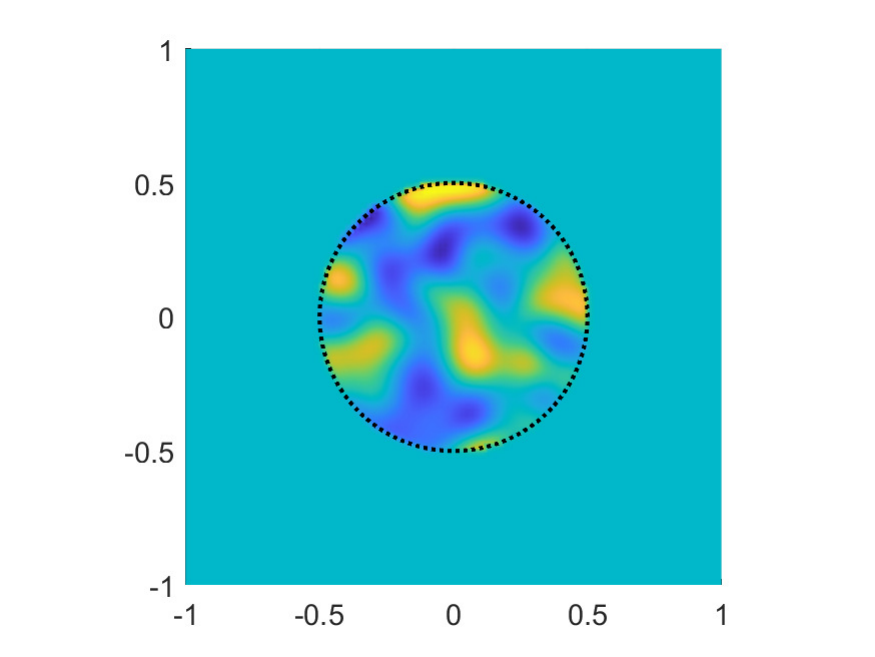}
    \put (50,-2) {$\displaystyle x$}
    \put (8,38) {$\displaystyle y$}
  \end{overpic}
\end{minipage}\hfil
\begin{minipage}{0.32\textwidth}
  \begin{overpic}[width=\textwidth]{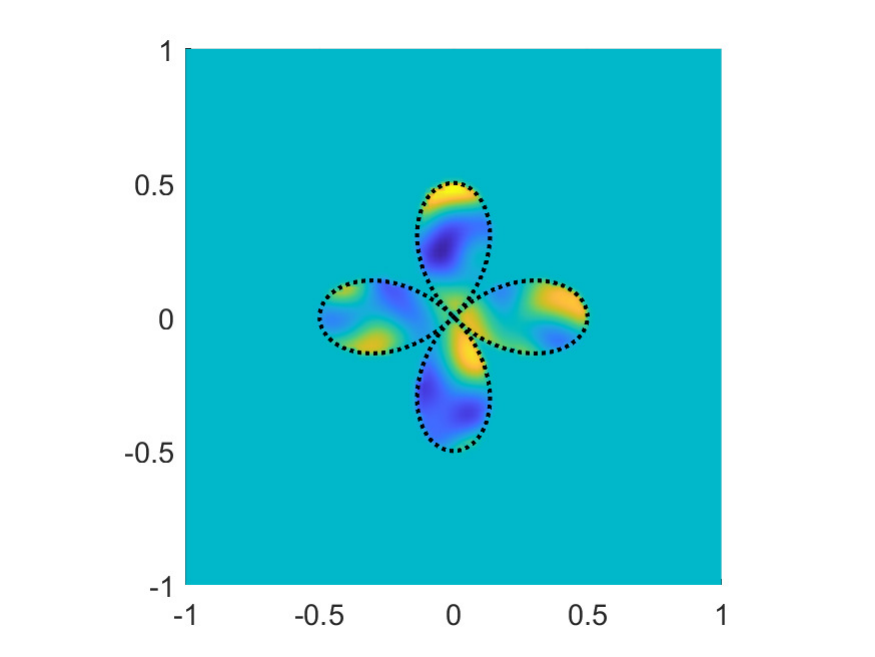}
    \put (50,-2) {$\displaystyle x$}
    \put (8,38) {$\displaystyle y$}
  \end{overpic}
\end{minipage}\hfil
\begin{minipage}{0.32\textwidth}
  \begin{overpic}[width=\textwidth]{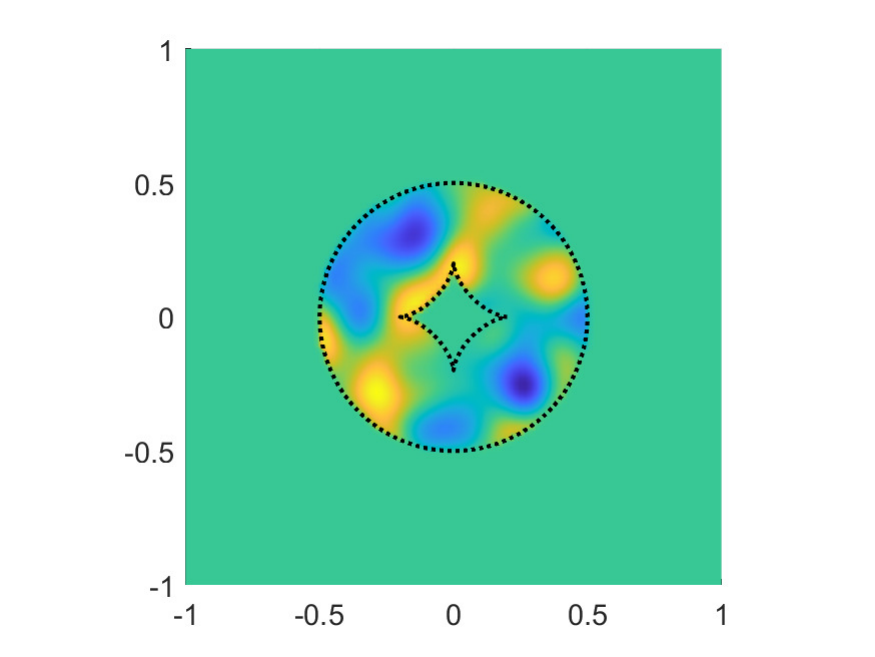}
    \put (50,-2) {$\displaystyle x$}
    \put (8,38) {$\displaystyle y$}
  \end{overpic}
\end{minipage}
\medskip
\begin{minipage}{0.32\textwidth}
  \begin{overpic}[width=\textwidth]{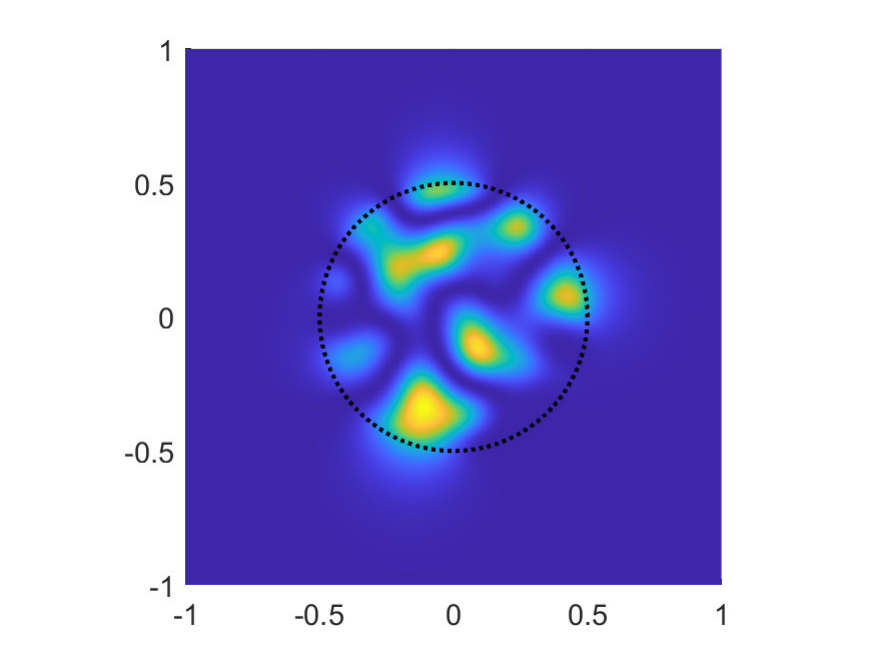}
    \put (50,-2) {$\displaystyle x$}
    \put (8,38) {$\displaystyle y$}
  \end{overpic}
\end{minipage}\hfil  
\begin{minipage}{0.32\textwidth}
  \begin{overpic}[width=\textwidth]{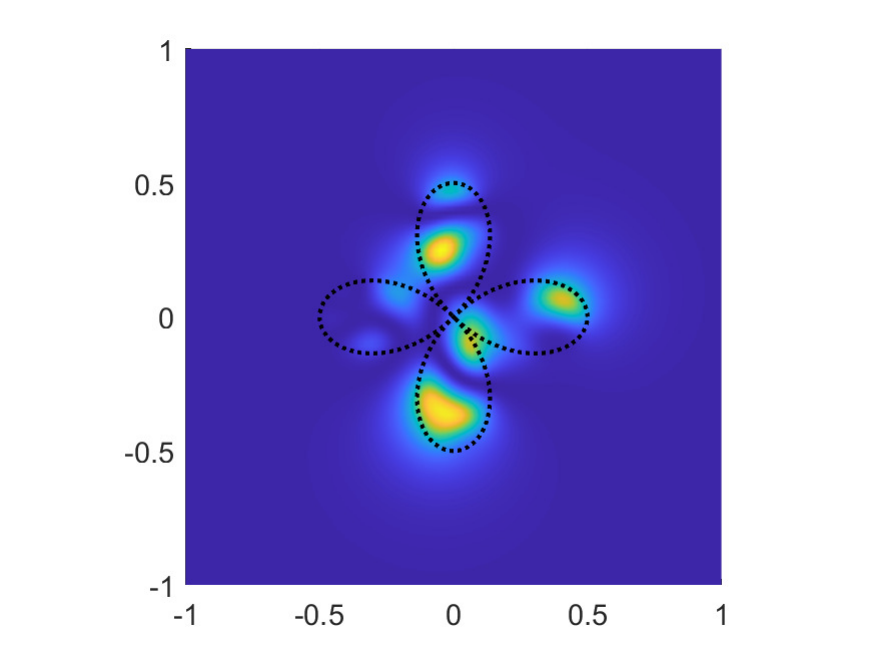}
    \put (50,-2) {$\displaystyle x$}
    \put (8,38) {$\displaystyle y$}
  \end{overpic}
\end{minipage}\hfil  
\begin{minipage}{0.32\textwidth}
  \begin{overpic}[width=\textwidth]{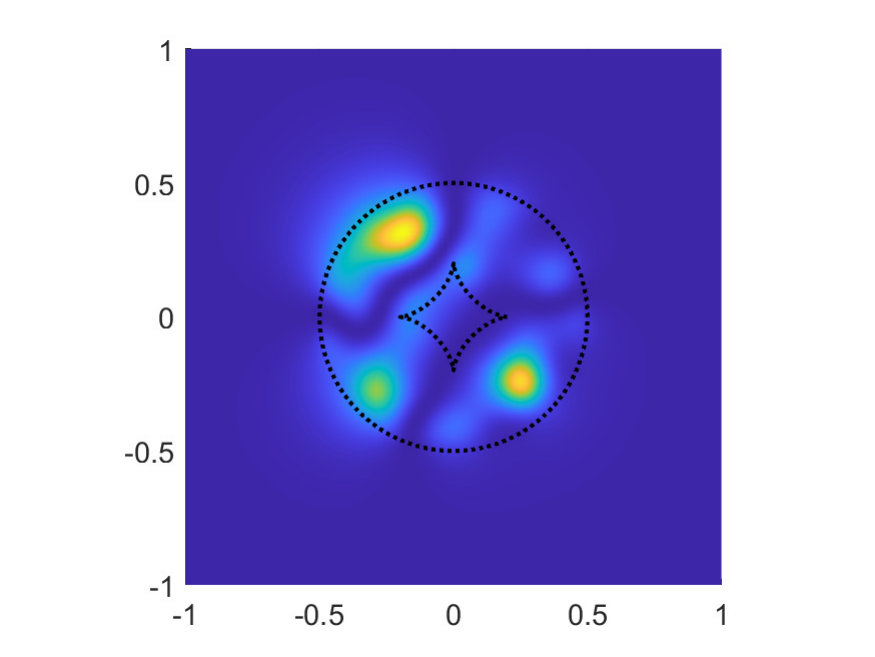}
    \put (50,-2) {$\displaystyle x$}
    \put (8,38) {$\displaystyle y$}
  \end{overpic}
\end{minipage}
\caption{Sample current sources $b(x,y)$ (top) and field intensities $|E_z(x,y)|^2$ (bottom), with frequency $\omega=\pi$, for the circle (left), quadrifolium (middle), and circle with an astroid cutout (right) cross-sections of diameter $0.5$. The trace estimates computed by $\ContHutch++$ are $\langle E_z\rangle\approx 0.0069$ (left), $0.0071$ (middle), and $0.0069$ (right). The boundary of $\Omega$ is marked with a dashed line in each panel.}
\label{fig:source_field}
\end{figure}
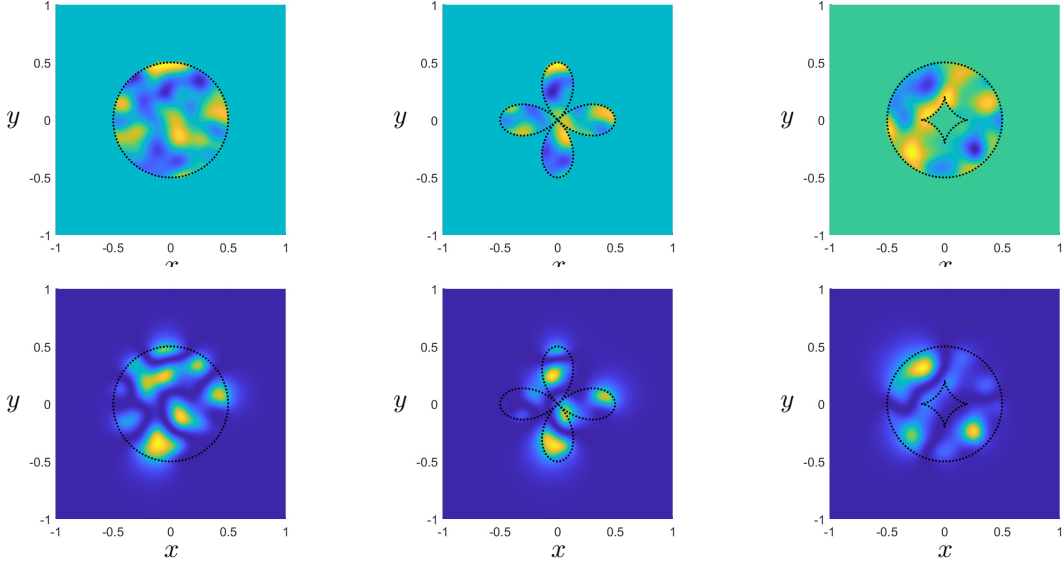

To approximate operator-function products with $H_\omega$, we discretize~\cref{eqn:Maxwell2D} with second-order centered finite differences on the unit square $[-1,1]^2\subset\mathbb{R}^2$, surrounded by a perfectly matched layer (PML) with unit thickness and strength, and prescribe homogeneous Dirichlet boundary conditions on the outer edges of the PML~\cite{oskooi2011distinguishing}. The dielectric cross-section $\Omega$ is scaled so that its diameter is inscribed in a disk of radius $r=0.5$ centered at the origin. Each numerical experiment in this section uses $100$ grid points in each direction on the unit square, and the interpolation parameter $\sigma>0$ in~\cref{eqn:smoothed_perm} is half the distance between grid points. The Gaussian process used for randomized trace estimation is formed from the tensor product of squared exponential kernels on the unit square with length-scale parameter $\ell=0.08$. The refinement of the $\Hutch++$ estimator with increasing sample size, $m$, is verified numerically for the circular cross-section in the right panel of~\cref{fig:spectrum_disk} with $\omega=0.5\pi$, $\pi$, and $2\pi$. \Cref{fig:source_field} displays a few random source currents, $b(x,y)$, used for trace estimation, the resulting field intensities, $|E_z(x,y)|^2$, and the estimated average field intensity $\langle E_z\rangle$ from~\cref{eqn:mean_field_intensity} for three dielectric cross-sections: a disk (left panel), a quadrifolium (middle panel), and a disk with an astroid cutout (right panel). The respective trace estimates computed by $\ContHutch++$ are $\langle E_z\rangle\approx 0.0069$ (disk), $0.0071$ (quadrifolium), and $0.0069$ (disk with astroid cutout).

\bibliographystyle{siam}
\bibliography{sample}

\appendix

\section{A Continuous Hanson--Wright Inequality} \label{sec:HW_appendix}

To prove an analogue of the Hanson--Wright inequality for operators, we first use the Karhunen–Loève theorem to write the functions drawn from a Gaussian process as a series of Gaussian random variables. 
This allows us to apply the discrete Hanson--Wright inequality to the first $n$ terms of the series. 
Finally, we take the  limit as $n\to \infty$ on the resulting inequality to get a bound for the infinite-dimensional case.\footnote{The authors thank the anonymous referee for suggesting this proof technique.}
 
Let $F$ be the integral operator defined in terms of $f$.
Let $K(x,y) = \sum_{i=1}^\infty \lambda_i e_i(x) e_i(y)$ be a Mercer decomposition of the covariance kernel $K$, where the functions $e_i$ form an orthonormal basis for $L^2(\Omega)$.
The Karhunen–Loève theorem tells us that a stochastic process can be expressed as a series of orthogonal functions weighted by random variables. 
In particular, if $g \sim \mathcal{GP}(0, K)$, then we can write $g = \sum_{i=1}^\infty \sqrt{\lambda_i} \gamma_i e_i$, where the $\gamma_i$ are i.i.d.~standard Gaussian random variables. Thus, if $R$ is the operator such that $R: (x_i) \in \ell^2 \to \sum_{i=1}^\infty \sqrt{\lambda_i} x_i e_i \in L^2(\Omega)$, then $g = R \gamma$, where we use $\gamma = (\gamma_i)$ as the sequence of standard Gaussians for convenience.

Using this representation of $g$, we can write $\phi(g)$ as 
\begin{equation*}
	\phi(g) = \langle R \gamma, F R \gamma \rangle_{L^2(\Omega)} = \langle \gamma, (R^T F R) \gamma \rangle_{\ell^2(\Omega)}. 
\end{equation*}
Since $F$ is positive definite, its eigendecomposition exists by Mercer's theorem. Let $\mu_i$ be the eigenvalues of $F$ and $\eta_i$ be the eigenvalues of $R^T F R$. 
Let $s(i)$ be a sequence such that $\mu_{s(1)} \geq \mu_{s(2)} \geq \cdots \geq 0$ is a reordering of the positive eigenvalues of $F$ in decreasing order.
Let $S_k$ denote any $k$-dimensional subset of $\ell^2(\Omega)$. 
Then the Courant-Fischer-Weyl min-max principle says that 
\begin{align*}
	\eta_{s(k)} &= \min_{S_{k-1}} \max_{\substack{\gamma \in S_{k-1}^\perp,\\ \|\gamma\|_{\ell^2(\Omega)} = 1}} \langle \gamma, (R^T F R) \gamma \rangle_{\ell^2(\Omega)} \\
	&= \min_{S_{k-1}} \max_{\substack{\gamma \in S_{k-1}^\perp,\\ \|\gamma\|_{\ell^2(\Omega)} = 1}} \langle (R\gamma), F (R\gamma) \rangle_{\ell^2(\Omega)} \\
	&= c_k \mu_{q(k)},
\end{align*}
where $c_k \in \mathbb{R}$ is some constant and the last line comes from the fact that $R\gamma$ is an eigenfunction of $F$, and thus must correspond to some eigenvalue which we denote $\mu_{q(k)}$. 
Note that the $R \gamma$ are eigenfunctions of $F$ with $\|R \gamma \|_{L^2} \leq \|R\|_{op} \|\gamma\|_{\ell^2} = \|R\|_{op} = \sqrt{\lambda_1}$. Thus, $c_k \leq \lambda_1$ for all $k \in \mathbb{N}$.
The same argument can be used to show $\eta_k = c_k \mu_q(k)$ with $c_k \in [0, \lambda_1]$ for zero eigenvalues. Therefore, the eigenvalues of $R^T F R$ are bounded by a positive constant multiple of the eigenvalues of $F$: $\eta_k = c_k \mu_k$, where $c_k \in [0, \lambda_1]$.

Thus, the eigendecomposition of $R^T F R$ can be written as $R^T F R (x) = \sum_{i=1}^\infty c_i \mu_i v_i \langle v_i, x \rangle_{\ell^2(\Omega)}$ for all $x\in \ell^2(\Omega)$, where the $v_i\in \ell^2(\Omega)$ are orthonormal eigenfunctions of $R^T F R$. Using this equation, we can express $\phi(g)$ as 
\begin{equation*}
	\phi(g) = \sum_{i=1}^\infty c_i \mu_i \langle v_i, \gamma \rangle^2_{\ell^2(\Omega)}.
\end{equation*}
However, since the $v_i$ are an orthonormal basis for $\ell^2(\Omega)$, the inner product $\langle v_i, \gamma \rangle_{\ell^2(\Omega)}$ is another independent standard Gaussian random variable. For convenience, we write $\tilde{\gamma}_i = \langle v_i, \gamma \rangle_{\ell^2(\Omega)}$. Then
\begin{equation*}
	\phi(g) = \sum_{i=1}^\infty c_i \mu_i \tilde{\gamma}_i^2
\end{equation*}
is a quadratic form on these standard Gaussian random variables. We can apply the discrete Hanson--Wright inequality~\cref{eq:discrete_HW} to the first $n$ terms of the series. Let $\phi_n(g) = \sum_{i=1}^n c_i \mu_i \tilde{\gamma}_i^2$. Then
\begin{align*}
	\mathbb{P}(|\phi(g) - \mathbb{E}[\phi(g)]| \geq t) &= \mathbb{P} \left(\lim_{n\to \infty} |\phi_n(g) - \mathbb{E}[\phi_n(g)]| \geq t \right)
\end{align*}
We want to bring the limit outside of the probability. 
For convenience, let $Z_n = |\phi_n(g) - \mathbb{E}[\phi_n(g)]|$ and $Z = \lim_{n\to \infty} Z_n$. 
Then we can rewrite this probability as $\int \mathbbm{1}_{\left(lim_{n\to \infty} Z_n \geq t \right)} dM$ for some probability measure $M$. 
Thus, we need to show that we can pull the limit out of the indicator function, $\mathbbm{1}$. Notice that if $Z > t$, then there exists some $N \in \mathbb{N}$ such that for all $n\geq N$, $|Z - Z_n| < Z - t$, so $Z_n > t$. 
In this case, $\lim_{n\to \infty} \mathbbm{1}_{\left(Z_n \geq t \right)} = \mathbbm{1}_{\left(Z \geq t \right)}$. 
The case of $Z<t$ is similar. 
Since $Z_n$ is defined in terms of Gaussian random variables, $\mathbb{P}(Z = t) = 0$. Thus, 
\begin{align*}
	\mathbb{P} \left(\lim_{n\to\infty} Z_n \geq t \right) &= \int \lim_{n\to\infty} \mathbbm{1}_{Z_n \geq t} dM \\
	&= \lim_{n\to \infty} \mathbb{P}(Z_n \geq t). 
\end{align*}
The second line comes from the dominated convergence theorem, since the indicator function is always dominated by $1$.

Then, by the discrete Hanson--Wright inequality,
\begin{align*}
	\mathbb{P} \left(|\phi(g) - \mathbb{E}[\phi(g)]| \geq t \right) &= \lim_{n\to \infty} \mathbb{P}(|\phi_n(g) - \mathbb{E}[\phi_n(g)]| \geq t) \\
	&\leq \lim_{n\to\infty} 2 \exp \left(- \frac{t^2}{16 \left( 5 \sum_{i=1}^n (c_i \mu_i)^2 + \max_{i=1, \ldots, n} (c_i \mu_i) t \right)}  \right) \\
	&\leq 2 \exp \left(- \frac{t^2}{16 \left( 5 \|F\|_{HS}^2 \|K\|_{op}^2 +  \|F\|_{op} \|K\|_{op} t \right)}  \right).
\end{align*}

\subsection{Proof of Continuous Hanson--Wright Inequality Corollary,~\cref{cor:HW}}
\begin{proof}
Notice that even when $f$ is not symmetric, we have
\begin{equation*}
    \int_\Omega \int_\Omega g(x) f(x,y) g(y) dx dy = \int_\Omega \int_\Omega g(x) f(y,x) g(y) dxdy. 
\end{equation*}
Thus, the symmetric function $\Tilde{f}(x,y) = \frac{1}{2}\left( f(x,y) + f(y,x) \right)$ satisfies 
\begin{equation*}
    \int_\Omega \int_\Omega g(x) \Tilde{f}(x,y) g(y) dx dy = \int_\Omega \int_\Omega g(x) f(x,y) g(y) dx dy.
\end{equation*}

Further, $\|\tilde{f}\| \leq \frac{1}{2}( \|f\| + \|f^T\|) = \|f\|$, for both the $L^2$ norm and the operator norm. Thus,~\cref{thm:HW} holds for both $f_1$ and $f_2$.
However, we want to find a bound for $f = f_1 - f_2$. Note that since $f = f_1 - f_2$, $\tilde{f} = \tilde{f_1} - \tilde{f_2}$ where $\tilde{f}_1(x,y) = \frac{1}{2}(f_1(x,y) + f_1(y,x))$ and $\tilde{f}_2 = \frac{1}{2}(f_2(x,y) + f_2(y,x))$, so we can express everything in terms of $\tilde{f_1}$ and $\tilde{f_2}$.
For convenience, let $\phi_1$ be defined in terms of $f_1$, $\phi_2$ be defined in terms of $f_2$, and $\phi$ be defined in terms of $f$. Thus, we can simplify the probability in terms of $f$ as
\begin{align*}
    \mathbb{P}\left( |\phi(g) - \E[\phi(g)]| \geq t  \right) &= \mathbb{P}\left( |\phi_1(g) - \phi_2(g) - \E[\phi_1(g) - \phi_2(g)]| \geq t  \right) \\
    &\leq \mathbb{P}\left( |\phi_1(g) - \E[\phi_1(g)]| + |\phi_2(g) - \E[\phi_2(g)]| \geq t  \right) \\
    &\leq \mathbb{P}(|\phi_1(g) - \E[\phi_1(g)]| \geq t/2) + \mathbb{P}(|\phi_2(g) - \E[\phi_2(g)] | \geq t/2),
\end{align*}
where the last line comes from the fact that to have $|\phi_1(g) - \E[\phi_1(g)]| + |\phi_2(g) - \E[\phi_2(g)]| \geq t$, we certainly need at least one of $|\phi_1(g) - \E[\phi_1(g)]|$ and $|\phi_2(g) - \E[\phi_2(g)]|$ to be at least $t/2$.

Let $\lambda_i$ be the $i^{th}$ eigenvalue of $\tilde{f}$. Let $S = \{ \lambda_i : \lambda_i \geq 0\}$.
Note that $\|\tilde{f}_j\|_{op} \leq |\lambda_1| = \|\Tilde{f}\|_{op}$ for $j=1,2$. Furthermore, $\|\tilde{f}_1\|_{L^2}^2 = \sum_{i\in S} \lambda_i^2 \leq \sum_{i=1}^\infty \lambda_i^2 = \|\Tilde{f}\|_{L^2}^2$ and $\|\tilde{f}_2\|_{L^2}^2 = \sum_{i\in \mathbb{N}-S} \lambda_i^2 \leq \sum_{i=1}^\infty \lambda_i^2 = \|\Tilde{f}\|_{L^2}^2$.
Thus, we can write the probability bound in terms of the norm of $f$. That is, from~\cref{thm:HW},
\begin{align*}
    \mathbb{P}\left( |\phi(g) - \E[\phi(g)]| \geq t  \right) &\leq 2 \exp \left(- \frac{t^2}{16 \left( 5 \|F\|_{L^2}^2 \|K\|_{op}^2 +  \|F\|_{op} \|K\|_{op} t \right)} \right) \\
    &+ 2 \exp \left(- \frac{t^2}{16 \left( 5 \|F\|_{L^2}^2 \|K\|_{op}^2 +  \|F\|_{op} \|K\|_{op} t \right)}  \right) \\
    &\leq 4 \exp \left(- \frac{t^2}{16 \left( 5 \|F\|_{L^2}^2 \|K\|_{op}^2 + \|F\|_{op} \|K\|_{op} t \right)}  \right).
\end{align*}
\end{proof}
\end{document}